\documentclass[10pt,a4paper, leqno]{article}
\usepackage[utf8]{inputenc}
\usepackage{amsmath}
\usepackage{amsfonts}
\usepackage{amssymb}
\usepackage{comment}
\usepackage{hyperref}
\usepackage{enumitem,pstricks,xy,pst-node}
\xyoption{all}
\usepackage{amsthm}
\usepackage{booktabs} 
\usepackage[color]{changebar} 
\cbcolor{red} 
\usepackage{mathrsfs} 
\usepackage{verbatim} 
\usepackage[disable] 
{todonotes}
\usepackage[normalem]{ulem} 
\usepackage{caption} 
\usepackage{nicefrac} 
\usepackage{enumitem} 

\usepackage{etoolbox}%

\patchcmd{\thebibliography}{\chapter*}{\par\let\clearpage\relax\chapter*}{\typeout{success}}{\typeout{failure}}

\usepackage{placeins} 

\newtheorem{theorem}{Theorem}[section]
\newtheorem{lemma}[theorem]{Lemma}
\newtheorem{proposition}[theorem]{Proposition}
\newtheorem{corollary}[theorem]{Corollary}
\newtheorem{definition}[theorem]{Definition}
\theoremstyle{definition}

\newtheorem{remark}[theorem]{Remark}
\newtheorem{notation}[theorem]{Notation}
\newtheorem*{ack}{Acknowledgements}

\theoremstyle{remark}
\newtheorem{example}[theorem]{Example}

\def\cU{{\mathcal U}}

\def\CC{\mathbb{C}}

\def\cO{{\mathcal O}}

\def\af1{\mathbf{aff}_1}

\DeclareMathOperator{\Chern}{Chern}

\newcommand{\ladi}{\begin{lastadd}}
\newcommand{\ladf}{\end{lastadd}}
\newcommand{\lrei}{\begin{lastrem}}
\newcommand{\lref}{\end{lastrem}}
\newenvironment{lastadd}
{\cbstart\color{red}}
{\todo{red to remove}\cbend}
\newenvironment{lastrem}
{\cbstart\color{yellow}}
{\cbend}

\author{Vladimiro Benedetti\thanks{Aix-Marseille Universit\'e, CNRS, Centrale Marseille, I2M, UMR 7373, 13453 Marseille, France.}}
\title{Manifolds of low dimension with trivial canonical bundle in Grassmannians}
\begin{document}
\maketitle

\begin{abstract}
We study fourfolds with trivial canonical bundle which are zero loci of sections of homogeneous, completely reducible bundles over ordinary and classical complex Grassmannians. We prove that the only hyper-K\"ahler fourfolds among them are the example of Beauville and Donagi, and the example of Debarre and Voisin. In doing so, we give a complete classification of those varieties. We include also the analogous classification for surfaces and threefolds.
\end{abstract}
\setcounter{tocdepth}{1}
\tableofcontents

\section{Introduction}
In Complex Geometry there are interesting connections between special varieties and homogeneous spaces. A striking evidence of this relation is the work of Mukai about Fano manifolds (as a reference, see \cite{Mukai10}). Mukai was able to reinterpret families of Fano threefolds as families of subvarieties of homogeneous spaces. His idea was that on a sufficiently general Fano threefold $X$, one can prove the existence of a certain positive vector bundle. Therefore, one gets a morphism, or eventually an embedding, from $X$ to a certain Grassmannian. It turned out that all the families of prime Fano threefolds (for which a classification was already known) admitted a nice description in terms of homogeneous bundles over Grassmannians. This description of the families was also helpful to understand better their geometry in relation to the well known geometry of the Grassmannians.

From then, other works have showed the richness of homogeneous spaces in providing examples of special varieties. For instance, it is generally a difficult problem to provide explicit families of hyper-K\"ahler manifolds, and few are known; among them, two can be seen as varieties of zeroes of a general global section of a homogeneous vector bundle over a Grassmannian. Both are fourfolds; the first one is the family of varieties of lines in a cubic fourfold, due to Beauville and Donagi (\cite{BeauvilleDonagi}). The variety of lines in a cubic fourfold is actually a subvariety of the Grassmannian of (projective) lines on a projective space of dimension five ($Gr(2, 6)$). It is the zero locus of a section of the third symmetric power of the dual of the tautological bundle. We denote this family of varieties by $X_1$. The second one, more recent, is due to Debarre and Voisin (\cite{VoisinDebarre}). They take the Grassmannian $Gr(6,10)$ of $6$-dimensional planes in a vector space $V$ of dimension ten, and consider a general skew symmetric $3$-form over $V$. The variety of planes isotropic with respect to this form is the zero locus of a section (which correspond to the form) of the third anti-symmetric power of the dual of the tautological bundle. They prove that this is a family, which we denote by $X_2$, of fourfolds which are hyper-K\"ahler.

These two examples motivated the present work. We study fourfolds which arise as zero loci of general global sections of homogeneous, completely reducible bundles over ordinary and classical Grassmannians. We will see that the only hyper-K\"ahler varieties of this form are those already mentioned; indeed, the following theorem holds:

\begin{theorem}
Suppose \emph{Y} is a hyper-K\"ahler fourfold which is the zero locus of a general section of a homogeneous, completely reducible, globally generated vector bundle over an ordinary or classical (symplectic or orthogonal) Grassmannian. Then either $Y$ is of type $X_1$ or of type $X_2$.
\end{theorem}

This theorem will be a direct consequence of the classification theorems of the following sections. For ordinary Grassmannians, we followed the analogous study done in \cite{Kuchle1995}, where the author has classified and then studied the properties of Fano fourfolds with index one obtained in the same way. Already in that case the two constraints for the varieties to be four dimensional and Fano of index one were sufficient to have a classification of the bundles which could give rise to the required varieties. 

We will generalize the result by giving a classification of fourfolds with trivial canonical bundle, following substantially the same ideas and proofs. With the help of the MACAULAY2-package SCHUBERT2 (\cite{Macaulay}) we will determine which subvarieties are Calabi-Yau (CY) and which are irreducible holomorphic symplectic (IHS, which is the same as hyper-K\"ahler) among the examples we have found. 

Then we will extend the classification to subvarieties of dimension four of the other classical Grassmannians. It should be remarked that, even though the symplectic and orthogonal Grassmannians can already be seen as varieties of zeroes of sections of homogeneous bundles over the ordinary Grassmannian, a new classification needs to be done. In fact, there are homogeneous bundles over the classical Grassmannians that are not restriction of bundles over the ordinary ones. For instance, the orthogonal of the tautological bundle is not irreducible, and one can quotient it by the tautological bundle. Moreover, the spin bundles in the orthogonal case do not extend to a bundle on the ordinary Grassmannian.

Finally we will present the corresponding results for dimension two and three. Some interesting varieties in this case have already been studied in detail. In particular, Mukai proved the unirationality of some moduli spaces of $K3$ surfaces with a given genus by giving an explicit locally complete family of $K3$s in a Grassmannian; those varieties are again zero locus of sections of homogeneous bundles (\cite{Mukai10}, \cite{Mukai13}, and \cite{Mukai1820}). We give the classification for surfaces and threefolds and, for the surfaces, we report also the computation of the degree (which gives the genus of the natural polarization of the surface) and the Euler characteristic. Surprisingly enough, there are many more cases, and they would be worth to be studied thoroughly.
\bigskip

As we were finishing the writing of this article, an article by D. Inoue, A. Ito and M. Miura which concerns the same subject has been published on arXiv (\cite{Japonais}). In this work the authors proved that, under the same hypothesis as ours, a finite classification is possible for subvarieties of the ordinary Grassmannian with trivial canonical bundle of any fixed dimension. They also study in more detail the case of CY threefolds, giving an explicit classification similar to ours and studying the cases found. On the other hand, they do not deal with the cases of symplectic and orthogonal Grassmannians, which is interesting too (for example, see Mukai's articles on $K3$ surfaces of genus seven and eighteen).

\begin{ack}
This work has been carried out in the framework of the Labex Archim\`ede (ANR-11-LABX-0033) and 
of the A*MIDEX project (ANR-11-IDEX-0001-02), funded by the ``Investissements d'Avenir" 
French Government programme managed by the French National Research Agency.\\
I would like to thank my advisor Laurent Manivel for useful and numerous suggestions. Moreover, I would like to thank Alexander Kuznetsov for suggesting to include the case of $OGr(n-1,2n)$ appearing in section \ref{ssOGr(n-1,2n)}, although this variety has Picard rank two.\\ 
\end{ack}

\section{Preliminaries}

In this section we recall some facts about homogeneous bundles over homogeneous varieties; for a more complete exposition see \cite{Ottaviani}. Afterwards, we recall basic definitions about varieties with trivial canonical bundle, and the Beauville-Bogomolov decomposition.

\subsection{Homogeneous bundles}
\label{prelhom}

Let $G$ be a reductive complex algebraic group (for instance, one of the classical groups $SL(n, \mathbb{C})$, $Sp(2n, \mathbb{C})$, $SO(2n+1, \mathbb{C})$ or $SO(2n, \mathbb{C})$). A variety $X$ is G-homogeneous if it admits a transitive algebraic (left) action of $G$. All homogeneous varieties can be seen as quotients $G/P$ of $G$ by a subgroup $P$. 
A homogeneous variety $G/P$ is projective and rational if and only if $P$ contains a Borel subgroup $B$ (it is the case of the Grassmannians); in this case the subgroup $P$ is said to be parabolic. Parabolic subgroups can be classified combinatorially.

A homogeneous vector bundle ${\cal F}$ over a homogeneous variety $X$ is a vector bundle which admits a G-action compatible with the one on the variety $X=G/P$.
If a vector bundle ${\cal F}$ is homogeneous, then the fiber ${\cal F}_{[P]}$ over the point $[P]\in X$ is stabilized by the subgroup $P$, i.e. ${\cal F}_{[P]}$ is a representation of $P$; the converse holds as well. More precisely, there is an equivalence of categories between homogeneous vector bundles over $G/P$ and representations of $P$. 
Therefore, in this context, one can define irreducible and indecomposable homogeneous bundles, in analogy with the definitions in representation theory. 

Note that $P$, contrary to $G$, is not reductive in general. Let $P_U$ be the unipotent factor of $P$ and $P_L$ a Levi factor. The latter is a reductive group. It turns out that a representation $\rho: P \to GL(V)$ (where $V$ is a vector space) is completely reducible if and only if $\rho|_{P_U}$ is trivial. So, completely reducible homogeneous bundles are identified with representations of $P_L$, and these in turn are identified with their maximal weights. This provides a combinatorial way to classify completely reducible homogeneous bundles which consists in indicating the maximal weights of the irreducible representations to which they correspond.

\subsection{Varieties with trivial canonical bundle}

There are three main categories of varieties which can be thought of as the "building blocks" of K\"ahler varieties with trivial canonical bundle, as suggested by the Beauville-Bogomolov decomposition theorem: complex tori, Calabi-Yau manifolds, and irreducible holomorphic symplectic manifolds. Complex tori of dimension $n$ are just compact quotients of $\CC^n$ by a lattice. Let us examine the other two classes in more detail. 

For what concerns Calabi-Yau manifolds, in the literature various different definitions can be found; we will use the following one:

\begin{definition}
\label{defCY}
A manifold $X$ with trivial canonical bundle is of Calabi-Yau type if $H^0(X, \Omega_{X}^0)\cong H^0(X, \Omega_{X}^{\dim(X)})\cong \mathbb{C}$, $H^0(X, \Omega_{X}^k)=0$ for $1\leq k \leq \dim(X)$, and its dimension is at least three.
\end{definition}

\begin{remark}
Calabi-Yau manifolds are manifolds of Calabi-Yau type which are simply connected.
\end{remark}

The condition on the dimension is required because K3 surfaces (which meet the requirements of Definition \ref{defCY}) are considered to be irreducible holomorphic symplectic surfaces. This brings us to the last class of varieties we consider:

\begin{definition}
A manifold $Z$ with trivial canonical bundle and dimension $2n$ is irreducible symplectic holomorphic, or hyper-K$\ddot{a}$hler, if $H^0(Z, \Omega_{Z}^*)=\mathbb{C}[\sigma] / \sigma^{n+1}$, where $\sigma\in H^0(Z, \Omega_{Z}^2)$ is everywhere nondegenerate.
\end{definition}

As anticipated before, these definitions gain in importance if we consider the following theorem (see \cite{Bogomolov}): 

\begin{theorem}[Decomposition theorem]
\label{decthm}
Let Y be a compact K$\ddot{a}$hler simply connected manifold with $K_Y={\cal O}_Y$. Then
\[
Y=\prod_i X_i\times\prod_j Z_j
\]
where 
\newline
-- $ X_i$ are simply connected Calabi-Yau manifolds;
\newline
-- $ Z_j$ are simply connected and irreducible holomorphic symplectic.
\end{theorem}

This result will be useful later to distinguish which of the varieties that we find are of Calabi-Yau type and which are irreducible holomorphic symplectic.

\section{Fourfolds in ordinary Grassmannians}

Let $Gr(k,n)$ be the Grassmannian of $k$-planes in a $n$-dimensional complex vector space. Let us denote by ${\cal U}$ the tautological bundle of rank $k$ (and ${\cal U}^*$ its dual), and by ${\cal Q}$ the tautological quotient bundle of rank $n-k$; $\Lambda^i(E)$ will denote the $i$-th exterior power of the bundle $E$, and $S^i(E)$ the $i$-th symmetric power of $E$; the ample generator of the Picard group of the Grassmannian (which corresponds to $det ({\cal U}^*)=det ({\cal Q})$) will be denoted by ${\cal O}(1)$, and ${\cal O}(n)={\cal O}(1)^{\otimes n}$. 

Our first main theorem is the following:

\begin{theorem}
\label{thm4ordin}
Let \emph{Y} be a fourfold with $K_Y={\cal O}_Y$ which is the variety of zeroes of a general section of a homogeneous, completely reducible, globally generated vector bundle ${\cal F}$ over Gr(k,n). Up to the identification of Gr(k,n) with Gr(n-k,n), the only possible cases are those appearing in Table \ref{tablefourord} in Appendix \ref{apptables}.
\end{theorem}

In the classification we have put also the computation of $\chi({\cal O}_Y)$ as it is the quantity that permits to distinguish between CY and IHS manifolds, as the former satisfy $\chi({\cal O}_Y)=2$ ($H^0(Y, \Omega^2_Y)=0$), while the latter satisfy $\chi({\cal O}_Y)=3$ ($H^0(Y, \Omega^2_Y)=\mathbb{C}$). 
\newline

All the subvarieties found are CY manifolds, with the exception of the cases $(d7)$, $(d5)$ and $(b12)$, which we examine now. The case $(b12)$ is the IHS fourfold appearing in \cite{BeauvilleDonagi}, while the case $(d7)$ is the one appearing in \cite{VoisinDebarre}. 

The case $ (d5)$ already appears in \cite{Reid}, where the variety of $n$-planes in the intersection of two quadrics in a space of dimension $2n+2$ is proved to be an abelian variety, the Jacobian variety of an hyperelliptic curve of genus $n+1$. So, for $(d5)$, $Y$ is an abelian variety.

On the other hand, the case $(d6)$ has $\chi({\cal O}_Y)=4$ because it is not connected. In fact it has two connected components, as if one considers the variety $Y_1$ of zeroes of a general section of $S^2{\cal U}^* $ in $Gr(4,8)$, it is the set of maximal subspaces isotropic with respect to a general symmetric 2-form. It is well known that $Y_1$ has two connected components. As a consequence, $(d6)$ represents a CY manifold with two connected components, which can be seen as complete intersections in the orthogonal Grassmannian $OGr(4,8)$. Note that this variety in turn is isomorphic to a quadric in $\mathbb{P}^7$, via the embedding given by ${\cal O}(\frac{1}{2})$.

\subsection{Classification}

The proof of the theorem will be divided into different lemmas and propositions which concern the subvarieties of $Gr(k,n)$ for different choices of $k$ and $n$.

\begin{notation}
The notations will be similar to those used in \cite{Kuchle1995}. The Grassmannian $Gr(k,n)$ will be thought of as the quotient $G/P_k$, where $P_k$ is the maximal parabolic subgroup containing the Borel subgroup of positive (standard) roots in $G=SL(n, \mathbb{C})$. Every irreducible homogeneous bundle is represented by its highest weight. A weight is represented by $\beta=(\beta_1,...,\beta_n)$ or by $(\beta_1,...,\beta_k;\beta_{k+1},...,\beta_n)$, where $\beta=\beta_1\lambda_1+\beta_2(\lambda_2-\lambda_1)+...+\beta_n(\lambda_n-\lambda_{n-1})$, and the $\lambda_i$'s are the fundamental weights for $G=SL(n, \mathbb{C})$. All weights can be renormalized so to have $\beta_n=0$.
\end{notation}

A consequence of the homogeneous condition is that as soon as a homogeneous bundle admits non zero global sections, it is globally generated. Another equivalent condition for a bundle to have global sections is the existence of a $G$-representation for the \emph{dual} of the weight representing the homogeneous bundle: in this case, the $G$-representation in question is canonically isomorphic to the space of global sections (see as a reference \cite{Bott}, or \cite[Theorem 11.4]{Ottaviani}).

\begin{remark}
\label{globgenord}
As we work with globally generated bundles, from now on the notation will change: to indicate a bundle with highest weight $\beta$ as before, we will write $(-\beta_k, ... , -\beta_1; \beta_{k+1}, ... , \beta_n)$, which is equivalent to taking the highest weight of the dual representation. In this way, a bundle $\alpha=(\alpha_1,...,\alpha_n)$ (according to the new notation) is globally generated when $\alpha_1\geq ...\geq\alpha_n\geq 0$, i.e. when the weight $\alpha$ is dominant under the action of $G$.
\end{remark}

\begin{example}
Over the Grassmannian \emph{Gr(3,7)} of 3-dimensional spaces in a 7-dimensional space, the dual tautological bundle $\cal{U}^*$ (of rank $3$) will be denoted by its highest weight $(1,0,0;0,0,0,0)=(1,0,...,0)$, the tautological quotient bundle $\cal{Q}$ by $(1,1,1;1,1,1,0)$, ${\cal O}(1) = \Lambda^3 {\cal U}^*$ by $(1,1,1;0,...,0)$. 
\end{example}

The rank of a bundle can be calculated explicitely:
\[
rank(\beta_1,...,\beta_n)=dim(\beta_1,...,\beta_k)\times dim(\beta_{k+1},...,\beta_n)
\] where
\[
dim(\beta_1,...\beta_r) = \prod\limits_{1\leq i<j\leq r} \frac{j-i+\beta_i-\beta_j}{j-i}
\]
is the Weyl character formula (see \cite[Chapter 24]{Fulton} for this formula and similar ones for classical groups).

The formula comes from the fact that the rank of a homogeneous bundle over $G/P$ is the same as the dimension of the $P$-module representing it (see Section \ref{prelhom}). As the $P$-module is irreducible, it is actually an irreducible module of the Levi factor of $P$, or of its Lie algebra, which is $sl(k)\oplus sl(n-k)$ for $G/P=Gr(k,n)$. Moreover the Weyl character formula gives the dimension of a $sl(r)$-module in terms of its highest weight. Putting these facts together, we obtain the formula for the rank in terms of the weight.

\begin{remark}
\label{remrank}
Analogous formulas for classical Grassmannians will be derived in a similar way; it will therefore be necessary to understand what is the Levi factor of $P$ for symplectic and orthogonal Grassmannians.
\end{remark}

The first three results have been proved in Kuchle's paper (\cite[Lemma 3.2, Lemma 3.4, Corollary 3.5]{Kuchle1995}). In what follows, we assume that the subvarieties of $Gr(k,n)$ we are constructing are zero loci of general sections of the homogeneous bundle ${\cal F}$. This bundle being completely reducible, it can be decomposed as a direct sum ${\cal F}= \oplus_i {\cal E}_i$.

\begin{lemma}
\label{lemma3.2}
One can assume that the following bundles over Gr(k,n) do not appear as summands in ${\cal F}= \oplus_i {\cal E}_i$:
\newline
(i) the bundles $(1,0,...,0)$ (respectively $(1,...,1,0)$) corresponding to ${\cal U}^*$ (resp. $Q$ on Gr(n-k,n) );
(ii) if $2k>n$ the bundles $(1,1,0,...,0)$ and $(2,0,...,0)$ (resp. $(1,...,1,0,0)$ and $(2,...,2,0)$ for $2k<n$).
\end{lemma}

One defines:
\[
dex(\beta)=(\frac{|\beta|_1}{k}-\frac{|\beta|_2}{n-k})rank(\beta)
\]
where $|\beta|_1=\sum_{i=1}^k \beta_i$ and $|\beta|_2=\sum_{i=k+1}^n \beta_i$.
\begin{lemma}
For ${\cal F}= \oplus_i {\cal E}_i$ we have
\begin{equation}
\label{eqrank}
rank({\cal F})=\sum_i rank({\cal E}_i)=k(n-k)-4\,\,\, ,
\end{equation}
and
\begin{equation}
\label{eqdex}
\sum_i dex({\cal E}_i)=n\,\,\, .
\end{equation}

\end{lemma}

\begin{proof}
The first formula of this lemma is the same as in \cite[Lemma 3.4]{Kuchle1995}, on the other hand the second is different as in this case one requires $K_Y={\cal O}_Y$.

However it is worth giving a hint on how to prove equation (\ref{eqdex}). Let us fix an irreducible $P$-module $V_{\beta}$ given by the weight $\beta$. Equation (\ref{eqdex}) is proved if we show that 
\[
det(V_{\beta})=L^{dex(V_{\beta})},
\]
where $L$ is the $1$-dimensional representation associated to the highest weight $(1,...,1;0,...,0)$. Now, the highest weight associated to $det(V_{\beta})$ is the sum of all the weights of $V_{\beta}$. Moreover, the space of weights is invariant under the action of the Weyl group $W$ of $P$, and for ordinary Grassmannians we have 
\[
W=\mathfrak{S}_k \times \mathfrak{S}_{n-k},
\] 
where $\mathfrak{S}_r$ stands for the symmetric group on $r$; the action on the weights is the expected one. Therefore, it is straightforward to see that the weight associated to $det(V_{\beta})$ is 
\[
(\dim(V_{\beta})\frac{|\beta|_1}{k},...,\dim(V_{\beta})\frac{|\beta|_1}{k};\dim(V_{\beta})\frac{|\beta|_2}{n-k},...,\dim(V_{\beta})\frac{|\beta|_2}{n-k}).
\]
By rescaling, this weight is equal to $(dex(\beta),...,dex(\beta);0,...,0)$, thus ending the proof
\end{proof}

\begin{remark}
\label{remdex}
The same method used to prove equation (\ref{eqdex}) can be used when dealing with symplectic and orthogonal Grassmannians. The only difference will be that generally one of the factors of the Weyl group of $P$ will be given by a certain \emph{signed} symmetric group (here again what matters is the Weyl group of the Levi factor of $P$).
\end{remark}

\begin{corollary}
Using the correspondence irreducible bundles - weights, in the same hypothesis as before,
\newline
(a) for each bundle ${\cal E}_i=(\beta_1^i,...,\beta_n^i)$ we have $\beta_1^i=...=\beta_k^i$ or $|\beta^i|_2=0$;
\newline
(b) $dim(\beta_1,...\beta_r)\geq {r \choose i}$ if $\beta_i>\beta_{i+1}$.
\end{corollary}

The strategy of the proof is the same as that of \cite{Kuchle1995}. We recall that the bundle ${\cal F}$ that defines the fourfold lives on $Gr(k,n)$, $2k\leq n$.

\begin{proposition}[Classification for $k\leq 3$]
\label{classkpetit}
If $k\leq 3$, then for ${\cal F}$ we have one of the cases labelled by the letters (a), (b), (c) appearing in Table \ref{tablefourord}.
\end{proposition}

\begin{proof}
$k=1$.

If $|\beta|_2\neq 0$, then $rank(\beta)\geq n-1>n-1-4$, so the only possible case is (a).
\newline
$k=2$.

$Gr(2,4)$ is a Fano variety, so one can suppose that $n\geq 5$. Calculating the rank, one can see that the only possible bundles ${\cal E}_i$ are $(p,q;0,...,0)$, $(r,r;1,0,...,0)$ for $r\geq 1$, $(s,s;1,...,1,0)$ for $s\geq 2$. If $(s,s;1,...1,0)$ is present as an addend, the only possibility is (b3) (if $s\geq 3$, then $dex(s,s;1,...1,0)\geq n$ ). For the same reason, if $(r,r;1,0,...,0)$ is present, $r=1$ and one has the cases (b10)(.i). Then, one has only the bundles $(p,q;0,...,0)$, for which $\frac{dex}{rank}\geq 1$, which forces $n\leq 8$. These are the remaining (b)-cases.
\newline
$k=3$.

One has $n\geq 6$, and, calculating the rank, for $n\geq 9$ the possible bundles are $(p,q,r;0,...,0)$ for $p\geq q\geq r$, $(p,q,r;0,...,0)\neq (1,0,...,0)$, $\neq (0,...,0)$; $(r,r,r;1,0,...,0)$ for $r\geq 1$; $(s,s,s;1,...,1,0)$ for $s\geq 2$. Supposing among the ${\cal E}_i$'s there are bundles different from those mentioned, verifying the possible ranks, no other cases arise for $n=6$ and $n=7$ (here one has to eliminate $(1,1,1;1,1,0,0)$ for \ref{lemma3.2}), while for $n=8$ one finds (c7) (again eliminating $(1,1,1;1,1,1,0,0)$). Now, if the bundle $(s,s,s;1,...,1,0)$ is present as an addend, one shows that (\ref{eqdex}) implies $s=2$, and that $(r,r,r;1,0,...,0)$ cannot appear. Moreover, as for $(p,q,r;0,...,0)$, $\frac{dex}{rank}\geq \frac{2}{3}$, one can only have the case (c1.1). If $(r,r,r;1,0,...,0)$ is present, equation (\ref{eqdex}) implies $r\leq 2$ and one easily eliminates also $r=2$ using equation (\ref{eqrank}). Therefore we remain with $r=1$; another bundle of the same kind would give $n\leq 8$ for (\ref{eqdex}), and $n\geq 7$ for (\ref{eqrank}). Using the remaining bundles $(p,q,r;0,...,0)$ and the relation they satisfy ($\frac{dex}{rank}\geq \frac{2}{3}$), one finds the cases (c2), (c2.1), (c5), (c5.1), (c7.1), (c7.2). Finally, if one has only bundles with $|\beta|_2=0$, then $\frac{dex}{rank}\geq \frac{2}{3}$ and $n\leq 8$. These are the remaining (c)-cases.
\end{proof}

\begin{lemma}["Reduction of cases"]
\label{lemmareduct}
If 
\begin{equation}
\label{eqreduct}
rank(\beta)(|\beta|_1 -1)\leq k^2+4
\end{equation}
and not considering the cases (d7) and (d8) in Table \ref{tablefourord}, then the possible bundles $\beta$ (representing ${\cal E}_i$) appearing as addend of ${\cal F}$ for $k, n-k\geq 4$ are the following:
\[
(A)=(1,...,1,0;0,...,0)
\]
\[
(B)=(1,1,0,...,0)
\]
\[
(C)=(2,0,...,0)
\]
\[
(Z)=(2,1,...,1;0,...,0)
\]
\[
(D)=(p,...,p;\beta_{k+1},...,\beta_n)
\]
\end{lemma}

\begin{remark}
\label{remlemmred}
If ${\cal F}= \oplus_i {\cal E}_i$, the expression $k^2+4$ is equal to $\sum_i (k(dex({\cal E}_i))-rank({\cal E}_i))= kn - kn+k^2+4$. So, if one knows that all the terms in the sum are positive, one can apply the lemma to each bundle appearing in ${\cal F}$. This is the case, as we will see, of Proposition \ref{classk=4}.
\end{remark}

\begin{proof}
Let us suppose that the bundle $\beta$ is not of type (D). If $\beta_k\geq 2$, then $rank(\beta)\geq k$ and $|\beta|_1-1\geq 2k$, which gives a contradiction with (\ref{eqreduct}). So, suppose $\beta_k=1$. If $\beta\neq $ (Z), it means that $rank(\beta)(|\beta|_1-1)> k(k+1)$, which is a contradiction with (\ref{eqreduct}). So one can suppose that $\beta_k=0$. If $|\beta|\geq 5$, as $rank(\beta)\geq \frac{k(k-1)}{2}$, we have a contradiction with (\ref{eqreduct}). If $|\beta|=4$, (\ref{eqreduct}) implies $k=4$, but by studying the possible cases one sees that none satisfies (\ref{eqreduct}). So $|\beta|\leq 3$. Apart from (B) and (C), there are cases (3,0,...,0) (not satisfying (\ref{eqreduct})), (2,1,0,...,0) (not satisfying (\ref{eqreduct})), and type (Y), i.e. (1,1,1,0,...,0). This last bundle satisfies (\ref{eqreduct}) only for $k=4,5,6$. 
For $k=4$, (Y)$=$(A). Let us study the possible appearances of (Y); they will be cases (d7) and (d8).
\newline
$k=6$.

For $n=10$, one has necessarily (d7). Suppose now $n\geq 11$ and that, in addition to (Y), there is a bundle $\beta$ of type (D). It must satisfy $rank(\beta)\leq 6n-60$ and $dex(\beta)\leq n-10$. It cannot be $(p,...,p;1,...,1,0)$ or $(p,...,p;1,0,...,0)$: as a consequence of (\ref{eqdex}), among them one is forced to the case $(1,...,1;1,...,1,0)$, which is not permitted by \ref{lemma3.2}. Then, $rank(\beta)\geq \frac{(n-6)(n-7)}{2}$, which contradicts (\ref{eqrank}). Therefore $|\beta|_2=0$, and $\beta$ satisfies the hypothesis of \ref{lemmareduct} (see remark after the lemma). But $k^2+4-(rank(Y)(|(Y)|_1 -1))=0$, so there is no room for any other summand, and then for any other case.
\newline
$k=5$.

Suppose that, in addition to (Y), there is a bundle $\beta$ of type (D). It must satisfy $rank(\beta)\leq 5n-39$, $dex(\beta)\leq n-6$. It cannot be $(p,...,p;1,...,1,0)$ or $(p,...,p;1,0,...,0)$ as a consequence of (\ref{eqdex}) as before. Then, $rank(\beta)\geq \frac{(n-5)(n-6)}{2}$, which is possible only for $n=9,10,11,12$. For $n=9$, one obtains (d8). For $n=10,11,12$, (\ref{eqdex}) and (\ref{eqrank}) show that there are no other solutions. Finally, if $|\beta|_2=0$, and $\beta$ satisfies the hypothesis of \ref{lemmareduct} (see remark after the lemma), calculating the quantity $\frac{dex}{rank}$ for the bundles (A),(B),(C),(Z),(Y) and the line bundles (D), one sees that it is always $\geq \frac{2}{5}$, which means that $\frac{n-6}{5n-39}\geq \frac{2}{5}$, or $n\leq 9$; but for $n=9$, one easily sees that no possibility arises if $|\beta|_2=0$. 
\end{proof}

\begin{proposition}[Classification for $k\geq 4$, $|\beta^i|_2=0$]
\label{classk=4}
Suppose $k,n-k\geq 4$. If the bundles appearing as summands in ${\cal F}= \oplus_i {\cal E}_i$ all satisfy $|\beta^i|_2=0$, then the only possible cases are (d1),(d2),(d3),(d4),(d6) of Table \ref{tablefourord}.
\end{proposition}

\begin{proof}
The hypothesis permits to apply \ref{lemmareduct} (see also Remark \ref{remlemmred} after the lemma). Then the only possible bundles are those mentioned: (A), (B), (C), (Z), (D). Let us define, for a bundle $\beta$, the quantity $\xi=\xi(\beta)=rank(\beta)(k(\frac{dex(\beta)}{rank(\beta)})-1)$.

\emph{If (Z) is a summand}, then $k^2+4-\xi(Z)=4$, and among the other bundles the only one which has $\xi\leq 4$ is (D) for $k=5,p=1$. But then $n=dex(Z)+dex(D)=k+2$, and this case is in \ref{classkpetit}.

\emph{If (A) is a summand}, we have $k^2+4-\xi(A)=4+2k$. Then another (A) cannot be added, otherwise, by computing $\xi$, $k=4$ and $n=7$, and this case is in \ref{classkpetit}. Similarly by adding (C). With (B), it is possible to have just $k=5$ and $n=9$, which is prohibited by \ref{lemma3.2}, or $k=4$ (cases (d1) and (d9)). Then, if there are just line bundles (D) in addition to (A), by studying $\xi$ we have $p\leq 2$, and one can see that all the cases arising have already been studied.

\emph{If (B) is a summand}, then $k^2+4-\xi(B)=4+\frac{k^2-k}{2}$. Suppose also that there is another (B) ($k^2+4-2\xi(B)=4$); then one can have only a bundle of type (D) with $p=2$ (case (d2)) or two bundles of type (D) with $p=1$ (case (d3)). If instead one supposes that there is a bundle of type (C), one finds the only possibility to be (d4). Finally, by supposing to add just line bundles, then their number must be $k(n-k)-4-rank(B)$, and the sum of their $dex$ must be equal to $n-dex(B)=n-k+1$. By imposing $\frac{dex}{rank}\geq 1$ (which is true for line bundles), and knowing that $n\geq 2k$ by \ref{lemma3.2}, one finds for $k$ the equation $k^2-k-10\leq 0$, which has no solution for $k\geq 4$.

\emph{If (C) is a summand}, then again $k^2+4-\xi(C)=4+\frac{k^2-k}{2}$; by adding another (C), one gets only the case (d5): $(2,0,...,0)^{\oplus 2}$ in Gr(4,10). Therefore, let us suppose the other bundles are only line bundles. One must have $0\leq rank({\cal F})-rank(C)\leq n-(dex(C))$ (for line bundles $\frac{dex}{rank}\geq 1$). Moreover, by \ref{lemma3.2}, one can suppose $n\geq 2$. Putting all together, one finds that the only possible case is (d6).

\emph{If there are only bundles of type (D)}, as $\frac{n}{k(n-k)-4}=\frac{dex}{rank}\geq 1$, and as $k, n-k\geq 4$, no other case arises.
\end{proof}

\begin{proof}[Proof of the classification]
As a consequence of \ref{classkpetit}, we can suppose $k, n-k\geq 4$. Using the isomorphism of $Gr(k,n)$ with $Gr(n-k,n)$, we suppose also $2k\geq n$. When $2k=n$, $\xi=rank(\beta)(k(dex(\beta))-1)=rank(\beta)((n-k)(dex(\beta))-1)$, and this symmetry says that all the bundles satisfy \ref{lemmareduct}. Then, dropping the hypothesis $|\beta|_2=0$, one only has to "symmetrize" the results found in \ref{classk=4}; this means that the cases (d2.1), (d3.1) and (d9.1) are to be added. 
 
So, from now on, $2k>n$. As the expressions of the form $(a,...,a;a,...,a)$ are not considered, and $\beta_i\geq \beta_{i+1}$, either $\frac{|\beta|_1-1}{k}\geq \frac{|\beta|_2}{n-k}$ or $\frac{|\beta|_1}{k}\geq \frac{|\beta|_2+1}{n-k}$, and in both cases $2k>n$ implies that all the terms of the sum on the right side of $k^2+4=\sum_i\xi({\cal E}^i)$ are positive. Then \ref{lemmareduct} can be applied. As we have \ref{classk=4}, and using $Gr(k,n)\longleftrightarrow Gr(n-k,n)$, one can suppose that there exists $i_0$ for which $|\beta^{i_0}|_2=0$ (and it is not a line bundle), but this doesn't hold for every $i$.

By \ref{lemma3.2}, this bundle must be either (A) or (Z). As $k\geq 5$, by computing $\xi$, one cannot have: (A) $\oplus$ (Z), (Z) $\oplus$ (Z), (A) $\oplus$ (A). For the bundles of type (D), let us change notation: 
\[
(p,...,p;\beta_{k+1},...,\beta_n) \longrightarrow (0,...,0; -\delta_1,..., -\delta_{n-k}) \longleftrightarrow (\delta_{n-k},...,\delta_1;0,...,0)
\]
where $\longleftrightarrow$ stands for $Gr(k,n)\longleftrightarrow Gr(n-k,n)$. Then, if $\beta^{i_0}=$ (Z), $k^2+4-\xi(Z)=4$; but the presence of a bundle $\delta$ which is of rank $\neq 1$ leads to a contradiction ($rank(\delta) (\frac{k}{n-k}|\delta|-1)> (n-k)(|\delta|-1)\geq 4(|\delta|-1)$, where $|\delta^i|=\sum_j \delta_j^i$).

As a result, the bundles present as summands of ${\cal F}$ are: one of type (A) and the others of type (D), with at least one which is not a line bundle. One has $k^2+4-\xi(A)=4+2k$. Then the condition $\sum_{i\neq i_0} (\xi(\beta^i))+\xi(A)=k^2+4$ becomes $\sum_{i\neq i_0} (\frac{k}{n-k}|\delta^i|-1)rank(\delta^i)=4+2k$. If for such a bundle which is not a line bundle $\delta_1\geq 1$, then $rank(\delta)(\frac{k}{n-k}|\delta| -1)>k(n-k)\geq 4k$, which is a contradiction. 

Therefore $\delta_1=0$. Define $\psi(\delta)=rank(\delta) (\frac{k}{n-k} |\delta|-1)$. If $|\delta|\geq 4$, $\psi(\delta)\geq 3k+1$. So, one is lead to consider $|\delta|=2, 3$. Among these bundles, with a similar estimate, one can eliminate $\delta=(2,1,0,...,0)$ and $ (3,0,...,0)$; $(1,1,1,0,...,0)$ is possible just for $n-k=4$. But then $n-dex(A)-dex(\delta)=2$, $k(n-k)-4-rank(A)-rank(\delta)=3k-8$, and one easily verifies that neither line bundles nor the bundles $\beta=(1,...,1,0,0),(2,...,2,0)$ can be added to give new cases. Therefore $\delta=2$ and coming back to the notation with $\beta$, the last cases to study are those with (A) $\oplus (1,...,1,0,0)$ or (A) $\oplus (2,...,2,0)$.

${\cal F}=$ (A) $\oplus (2,...,2,0)\oplus ...$ In this case $n-dex(A)-dex(2,...,2,0)=0$, so there cannot be other bundles. Then equation (\ref{eqrank}) gives 
\[
0=4kn-n^2-3k^2-n-k-8=-(n-2k)^2+k^2-n-k-8=k^2-a^2+a-8-3k
\]
where $a=2k-n$. Integer solutions for $k$ are given only if $4a^2-4a+41=b^2=c^2+40$, where $c=2a-1$, and $b$ is an integer. By writing down all the integer solutions for $(b+c)(b-c)=40$, none gives new cases.

${\cal F}=$ (A) $\oplus (1,...,1,0,0)\oplus ...$ In this case $n-dex(A)-dex(1,...,1,0,0)=2$, so there cannot be summands other than (A), $(1,...,1,0,0)$ and line bundles. Then $k(n-k)-4-rank(A)-rank(1,...,1,0,0)$ can be only $2$ or $1$, and equation (\ref{eqrank}) gives 
\[
0=4kn-n^2-3k^2+n-3k-c=-(n-2k)^2+k^2+n-3k-c=k^2-a^2-a-c-k
\]
where $a=2k-n$, and $c$ can be $10$ or $12$. Integer solutions for $k$ are given only if $4a^2+4a+4c+1=b^2=d^2+4c$, where $d=2a+1$, and $b$ is an integer. By writing down all the integer solutions for $(b+d)(b-d)=4c$, for $c=10,12$, none gives new cases.
\end{proof}

\begin{remark}
\label{remsmoothness}
Having dealt with the combinatorics of the problem, we turn to the geometry. All the bundles we have considered so far (and appearing in Table \ref{tablefourord}) are globally generated (Remark \ref{globgenord}). Therefore, by applying the usual Bertini theorem, our subvarieties with trivial canonical bundle are smooth. The same will hold when dealing with the classical Grassmannians, as in that case too all the bundles considered will be globally generated.
\end{remark}

\subsection{CY vs IHS}
\label{secCYvsIHS}

We want now to show how to distinguish between CY and IHS manifolds in an efficient way (we have to deal with a great number of cases). The following proposition holds:

\begin{proposition}
\label{propCYvsIHS}
Suppose $Y$ is a smooth projective fourfold with trivial canonical bundle. If the Euler characteristic of the trivial bundle $\chi({\cal O}_Y)$ is either two or three, then $Y$ is simply connected. Moreover:
\newline
-- If $\chi({\cal O}_Y)=2$, then $Y$ is CY;
\newline
-- If $\chi({\cal O}_Y)=3$, then $Y$ is IHS;
\end{proposition}

\begin{proof}
Let us first suppose to have proven that $Y$ is simply connected. Then we can apply Theorem \ref{decthm}. Therefore, our variety $Y$ is a product of CY and IHS manifolds, i.e. it is either a product of two K3 surfaces, or a CY fourfold, or a IHS one. In the first case $\chi({\cal O}_Y)=4$, in the second $\chi({\cal O}_Y)=2$ and in the third $\chi({\cal O}_Y)=3$, thus proving the last assertion.

Next we turn to the proof of simply connectedness of $Y$. As a matter of fact, a generalization of the decomposition Theorem \ref{decthm} holds (\cite{Beauville}): for any compact K$\ddot{a}$hler manifold $Y$ with trivial canonical bundle, there exists a finite cover $f: Y' \to Y$ of degree $d$ such that 
\[
Y'=\prod_i X_i\times\prod_j Z_j\times\prod_k T_k
\]
where $X_i$ are simply connected CY's, $Z_j$ are simply connected IHS', and $T_k$ are complex tori. Moreover, a well-known formula says:
\[
\chi(Y', {\cal O}_{Y'})= d \,\, \chi(Y, {\cal O}_{Y})
\]
Recall also that the Euler characteristic of a product is the product of the Euler characteristics of the single factors, and that $\chi(\cO_T)=0$ for $T$ a complex torus. Then, the only possibilities for $Y'$ are:
\newline
-- $Y'$ is CY and $\chi({\cal O}_{Y'})=2$;
\newline
-- $Y'$ is IHS and $\chi({\cal O}_{Y'})=3$;
\newline
-- $Y'=K3\times K3$ and $\chi({\cal O}_{Y'})=4$.
\newline
-- $Y'=CY\times E$ where $E$ is an elliptic curve and $\chi({\cal O}_{Y'})=0$;
\newline
-- $Y'=K3\times T$ where T is a complex torus of dimension two and $\chi({\cal O}_{Y'})=0$;
\newline
-- $Y'$ is a product of complex tori and $\chi({\cal O}_{Y'})=0$.
\smallskip

Therefore, if $\chi({\cal O}_Y)$ is either two or three, then for such a cover $f: Y' \to Y$ the degree is one ($f$ is an isomorphism), and $Y$ is either CY or IHS, and in both cases it is simply connected.
\end{proof}

\begin{remark}
Notice that connectedness too is ensured by the hypothesis that $\chi(\cO_Y)=2$ or $3$. In fact, the above proof shows that $\chi(\cO_Y)$ cannot be equal to one.
\end{remark}

For the actual computation of $\chi({\cal O}_Y)$ it is possible to use the \emph{Hirzebruch-Riemann-Roch theorem}, which gives:
\[
\chi({\cal O}_Y)=\int_Y td (Y)
\]
where $td(Y)=todd (TY)$ is the todd class of the tangent bundle. SCHUBERT2 allows to compute easily these quantities for subvarieties of Grassmannians.

\begin{example}
\emph{Here we report an example of computation of $\chi({\cal O}_Y)$ for the case (c6):}
\newline

$ G=flagBundle({3,4}) $ \newline
 
 $ listtautologicalbundle=$ $bundles$ $G$ \newline
 
 $ Tstar=dual(listtautologicalbundle\# 0) $ \newline
 
 $ O(1)= det$ $Q $ \newline
 
 $ F=exteriorPower(2, Tstar)+exteriorPower(2, Tstar)+O(1)+O(1)*O(1)$ \newline
 
$Y=sectionZeroLocus$ $F$ \newline

 $ Tan=tangentBundle$ $Y$ \newline
 
 $ td= todd$ $Tan$ \newline
 
 $chi = integral$ $td$ \newline
 
 \end{example}
\medskip

Another aspect of these varieties that can be studied is their Hodge numbers. A tool which is useful in this sense is the Koszul complex for a variety $Y$ which is the zero variety of a section of a vector bundle ${\cal F}$ over another variety G. If the bundle has $rank=r$, and $codim_G(Y)=r$, then one has the exact sequence:
\[
0 \to \Lambda^r {\cal F}^* \to \Lambda^{r-1} {\cal F}^* \to ...  \to \Lambda^2 {\cal F}^* \to {\cal F}^*\to {\cal O}_G \to {\cal O}_Y \to 0
\]
Through this complex, tensoring it by any other bundle, it is possible to find the cohomology groups of the restriction of the bundle to Y. Moreover, one can use the short exact sequence
\[
0 \to {\cal F}^*|_Y \to \Omega_G^1|_Y \to \Omega_Y^1 \to 0
\]
to study the cohomology groups of the cotangent bundle of Y. This is not enough in general; one needs to know the cohomology groups on the variety G. But for this it is possible, as G=Gr(k,n) and ${\cal F}$ is homogeneous, to use Bott's theorem (\cite{Bott}, \cite[Theorem 2.3]{Kuchle1995} for the version that is needed here).

\begin{example}
\emph{It is a (lengthy) exercice to compute the Hodge Diamond of case (c6); the result is displayed in Table \ref{tablec6}.} 
\begin{table}
$$
\xymatrix@-1.5pc{&&&&1&&&&\\ 
\rnode{A3}{}&\rnode{A5}{}&&0&&0&&\rnode{B5}{}&\rnode{A4}{} \\ 
&&0&&1&&0&&\\ 
&0&&0&&0&&0&\\ 
1&&145&&628&&145&&1\\
&0&&0&&0&&0&\\
&&0&&1&&0&&\\
\rnode{B3}{}&\rnode{A6}{}&&0&&0&&\rnode{B6}{}&\rnode{B4}{} \\ 
&&&&1&&&&}
$$
\caption{Hodge diamond of case (c6)}
\label{tablec6}
\end{table}
\end{example}

However, this method takes some time to be employed, though it is not complicated using the Littlewood-Richardson rule. Therefore we decided not to include computations of cohomology groups apart from when necessary.

\section{Fourfolds in classical Grassmannians}

In this section we study subvarieties of classical (symplectic and orthogonal) Grassmannians, showing that a result similar to Theorem \ref{thm4ordin} holds. The geometry of these isotropic Grassmannians is well known (see \cite[Chapter 6]{FultonPragacz} for some basic properties). They belong to the class of \emph{flag manifolds}, for which a good reference is \cite{Brion}. Their (classical and quantum) cohomology has been the subject of several studies (for instance, see \cite{Buchtam}).\newline

Let $IGr(k,2n)$ ($OGr(k,m)$, $m=2n, 2n+1$) be the symplectic (orthogonal) Grassmannian of isotropic - with respect to an anti-symmetric (symmetric) non-degenerate form - $k$-planes in a $2n$-dimensional ($2n$, $2n+1$ - dimensional) complex vector space. Let us denote ${\cal U}$ the tautological bundle of rank $k$ (and ${\cal U}^*$ its dual), and ${\cal Q}$ the tautological quotient bundle of rank $2n-k$; $\Lambda^i(E)$ will denote the i-th exterior power of the bundle $E$, and $S^i(E)$ the i-th symmetric power of $E$; the ample generator of the Picard group of the Grassmannian (which corresponds to $det ({\cal U}^*)=det ({\cal Q})$) will be denoted ${\cal O}(1)$, and ${\cal O}(n)={\cal O}(1)^{\otimes n}$. 

Moreover, over $IGr(k,2n)$, ${\cal U}^{\perp}$ will denote the orthogonal of the tautological bundle (of rank $2n-k$), which is indecomposable but not irreducible: in fact, there is an injective homomorphism ${\cal U}\to {\cal U}^{\perp}$, and the quotient ${\cal U}^{\perp}/ {\cal U}$ is irreducible of rank $2n-2k$.

Over $OGr(k,2n+1)$, ${\cal T}_{+\frac{1}{2}}$ will denote the spin bundle of rank $2^{n-k}$, and over $OGr(k,2n)$, ${\cal T}_{+\frac{1}{2}}$ and ${\cal T}_{-\frac{1}{2}}$ will denote the two spin bundles of same rank $2^{n-k-1}$. Finally, over $OGr(n,2n)$, the line bundle ${\cal O}(1)$ is not a generator of the Picard group. It is actually divisible, and its square root will be denoted ${\cal O}(\frac{1}{2})$ (note that over $OGr(n,2n+1)$, ${\cal T}_{+\frac{1}{2}}$ is again a square root of ${\cal O}(1)$). ${\cal O}(\frac{1}{2})$ is the line bundle which gives the spinorial embedding of $OGr(n,2n)$ in $\mathbb{P}^{2^{n-1}-1}$.

The main theorem is the following.
\begin{theorem}
\label{thm4class}
Let \emph{Y} be a fourfold with $K_Y={\cal O}_Y$ which is the variety of zeroes of a general section of a homogeneous, completely reducible, globally generated vector bundle ${\cal F}$ over the symplectic Grassmannian IGr(k, 2n) (respectively the odd orthogonal Grassmannian OGr(k,2n+1), the even orthogonal Grassmannian OGr(k,2n)), and which does not appear in the  analogous classification for the ordinary Grassmannian. Up to identifications, the only possible cases are those appearing in Table \ref{tablefoursym} (respectively Table \ref{tablefourodd}, Table \ref{tablefoureven}).
\end{theorem}

\begin{remark}
The varieties $Y$ appearing in Theorem \ref{thm4class} are smooth (see Remark \ref{remsmoothness}).
\end{remark}

\begin{remark}
All the cases studied refer to subvarieties of classical Grassmannians which are the quotient of a classical group $G$ by a parabolic subgroup associated to a single simple root. In fact usually "Grassmannians" refer to these quotients. So, we have skipped the classification of subvarieties of $OGr(n-1,2n)$, because in this case the corresponding parabolic subgroup is associated to the last two simple roots of the Dynkin diagram $D_n$. However, for the sake of completeness we have also reported the analogous classification for $OGr(n-1,2n)$ at the end of this section.
\end{remark}

\begin{remark}
It is well known that the Grassmannians $OGr(n-1,2n-1)$ and $OGr(n,2n)$ are isomorphic. But the bundles which are homogeneous in one case may not be homogeneous in the other. For example, consider $\Lambda^2 {\cal U}^*$ on $OGr(n,2n)$, which is the tangent bundle. Pulling back this bundle via the isomorphism gives the tangent bundle $T$ on $OGr(n-1,2n-1)$, which is not a priori the second exterior power of a vector bundle homogeneous with respect to $so(2n-1)$, and is not irreducible. On the contrary, ${\cal O}(1)$ on $OGr(n,2n)$ pulls back to the corresponding ${\cal O}(1)$ on $OGr(n-1,2n-1)$, and the same for ${\cal O}(\frac{1}{2})$. So, referring to Table \ref{tablefourodd} and Table \ref{tablefoureven}, one can easily identify cases $(oz3)$ and $(oy1)$, $(oz4)$ and $(oy5)$, $(oz5)$ and $(oy4)$, $(oz7)$ and $(oy1.1)$.
\end{remark}

We break the classification given by Theorem \ref{thm4class} into three parts, which correspond to subvarieties in symplectic, odd and even orthogonal Grassmannians. Furthermore, the method used to understand if the varieties we have found are CY or IHS is the one already used in Section \ref{secCYvsIHS}. Indeed, we want to apply Proposition \ref{propCYvsIHS}. In order to do so, we need to compute the Euler characteristic of the trivial bundle of the variety. This requires some technical facts about the cohomology of classical Grassmannians; in Appendix \ref{appEulchar} we reported the details of how to do such a computation.

\subsection{Symplectic Grassmannians}


The symplectic Grassmannian $IGr(k,2n)$ will be thought of as the quotient $G/P_k$, where $P_k$ is the maximal parabolic subgroup containing the standard Borel subgroup of positive roots in $G=Sp(2n, \mathbb{C})$. Every irreducible homogeneous bundle is represented by its highest weight. A weight is represented by $\beta=(\beta_1,...,\beta_n)$ or by $(\beta_1,...,\beta_k;\beta_{k+1},...,\beta_n)$ when this notation is needed, where $\beta=(\beta_1-\beta_2)\lambda_1+\lambda_2(\beta_2-\beta_3)+...+\lambda_{n-1}(\beta_{n-1}-\beta_n)+\lambda_n \beta_n$, and the $\lambda_i$'s are the fundamental weights for $G=Sp(2n, \mathbb{C})$. Notice that the parabolic algebra $Lie(P_k)$ has Levi factor $sl(k)\oplus sp(2(n-k))$ in general ($k\neq 1, n$), which is straightforward by looking at the Dynkin diagram.

\begin{notation}
As we work with globally generated bundles, from now on the notation will change: to indicate a bundle with highest weight $\beta$ as before, we will write $(-\beta_k, ... , -\beta_1; \beta_{k+1}, ... , \beta_n)$, which is equivalent to taking the highest weight of the dual representation. In this way, a bundle $\alpha=(\alpha_1,...,\alpha_n)$ (according to the new notation) is globally generated when $\alpha_1\geq ...\geq\alpha_n\geq 0$, i.e. when the weight $\alpha$ is dominant under the action of $G$. To understand why we use this "dualized" notation, refer to the explication before Remark \ref{globgenord}.
\end{notation}

\begin{example}
Over the symplectic Grassmannian \emph{IGr(3,14)}, the dual tautological bundle ${\cal U}^*$ (of rank $3$) will be denoted by $(1,0,0;0,0,0,0)=(1,0,...,0)$, and the tautological "orthogonal" bundle ${\cal U}^{\perp}/{\cal U}$ (of rank $8$) by $(0,0,0;1,0,0,0)$. 
\end{example}

The dimension of a bundle can be calculated explicitely: suppose $k\neq 1, n$;
\[
rank(\beta_1,...,\beta_n)=dim_{sl(k)}(\beta_1,...,\beta_k)\times dim_{sp(2(n-k))}(\beta_{k+1},...,\beta_n)
\] where
\[
dim_{sl(r)}(\beta_1,...\beta_r) = \prod\limits_{1\leq i<j\leq r} \frac{j-i+\beta_i-\beta_j}{j-i},
\]
and
\[
dim_{sp(2r)}(\beta_1,...\beta_r) = \prod\limits_{1\leq i<j\leq r} \frac{j-i+\beta_i-\beta_j}{j-i} \prod\limits_{1\leq i\leq j\leq r} \frac{2r+2-j-i+\beta_i+\beta_j}{2r+2-j-i}
\]
are the Weyl character formula relative to the corresponding Lie algebras (see \cite[Chapter 24, Equation 24.19]{Fulton}). This formula is a consequence of the form of the Levi factor of $P$ and of Remark \ref{remrank}.

One defines:
\[
dex(\beta)=(\frac{|\beta|_1}{k})rank(\beta)
\]
where $|\beta|_1=\sum_{i=1}^k \beta_i$. Then, similarly to the case for the ordinary Grassmannian (see Remark \ref{remdex}), 
\[
det (\beta)={\cal O}(dex(\beta))
\]
Therefore, fourfolds with trivial canonical bundle correspond to homogeneous vector bundles ${\cal F}=\oplus_i {\cal E}_i$, with
\begin{equation}
\label{symprank}
\sum_i rank({\cal E}_i)=k(2n-k)-\frac{k(k-1)}{2}-4\,\,\, ,
\end{equation}
\begin{equation}
\label{sympdex}
\sum_i dex({\cal E}_i)=2n-k+1\,\,\, .
\end{equation}

\subsubsection{Classification in symplectic Grassmannian}

We recall that $IGr(k, 2n)$ is embedded naturally in $Gr(k, 2n)$ as the zero locus of a general section of $\Lambda^2 \cU^*$. The following lemma is to avoid repeating cases already considered in Theorem \ref{thm4ordin} and will be used throughout the proof of the classification.

\begin{lemma}
\label{lemmaatleast}
Suppose ${\cal F}=\oplus_i {\cal E}_i$ is as in the hypothesis of Theorem \ref{thm4class}. In order to find new varieties in IGr(k,2n) with respect to the case of the ordinary Grassmannian, it is necessary that for at least one bundle ${\cal E}_i=\beta_i$, $|\beta_i|_2 \neq 0$
\end{lemma}

\begin{proof}
The tautological bundle over $IGr(k,2n)$ is the restriction of the tautological bundle over $Gr(k,2n)$. Similarly, the bundle represented by the weight $\beta=(\beta_1,...,\beta_k,0,..,0)$ over the symplectic Grassmannian is the restriction of the bundle represented by the same weight $\beta$ over the ordinary Grassmannian. Therefore if for all $i$, $|\beta_i|_2 = 0$, the resulting fourfold is already the zero locus of a homogeneous bundle over $Gr(k,2n)$; as a consequence it has already been considered in Theorem \ref{thm4ordin}
\end{proof}

So, one has to suppose that $k\neq n$. Moreover, if $k=1$, $IGr(1, 2n)=Gr(1, 2n)$, so one can also suppose $k \neq 1$.

One can assume that $({\cal U}^*)^{\oplus 2}$ over \emph{IGr(k,2n)} does not appear as summand in ${\cal F}= \oplus_i {\cal E}_i$;
in fact, taking a zero section locus of this bundle in $IGr(k, 2n)$ is equivalent to restricting to the space $IGr(k, 2(n-1))$. 

Finally, remark that for any bundle ${\cal E}$ globally generated summand of ${\cal F}$ (even for the odd and even orthogonal Grassmannians), 
\begin{equation}
\label{sympineqdexrank}
\frac{dex({\cal E})}{rank({\cal E})}=\frac{|\beta|_1}{k}\geq \frac{1}{k} \,\,\, .
\end{equation}
\newline

The proof now consists in studying cases with low $k$ (in $IGr(k,2n)$), and then eliminating any other possibility.
\begin{proposition}[Classification for $k\leq 3$]
\label{classkpetitsym}
If $k\leq 3$, as suitable ${\cal F}$ (satisfying the hypothesis of Theorem \ref{thm4class}) one has all and only the bundles appearing in Table \ref{tablefoursym}.
\end{proposition}

\begin{proof}
$k=2$

In this case $rank({\cal F})=4n-9$, and $dex({\cal F})=2n-1$. If one has the bundle $\lambda$, with $\lambda_1 \neq \lambda_2$ and $|\lambda|_2\neq 0$, then $rank(\lambda)\geq 2(2(n-2))=4n-8$, which is impossible; so, by checking the dimensions of the corresponding modules (and comparing with equations (\ref{symprank}) and (\ref{sympdex})), one remains with the bundles $(p,q;0,...,0)$ for ($p\geq q$), $(p,p;1,0,...,0)$ for $n\geq 3$, and finally for $n=4$ with those mentioned and with $(p,p;1,1)$. Then one checks that this gives the bundles $(sb)$ for $n=3, 4$. For $n>4$, one knows that there must be one summand of the form $(p,p;1,0,...,0)$, and one sees that eq. (\ref{sympdex}) implies $p=1$. But then 
$$
\frac{dex({\cal F})-dex((1,1;1,0,...,0))}{rank({\cal F})-rank((1,1;1,0,...,0))}=\frac{3}{2n-5}\geq \frac{1}{2}
$$ 
by (\ref{sympineqdexrank}), which means $n=5$, for which one can check by hand that there is no other possibility.
\newline
$k=3$

In this case $rank({\cal F})=6n-16$, and $dex({\cal F})=2n-2$. Doing the computation by hand, for $n=4$ one finds all the cases $(sc)$. For a bundle $\beta$ such that there exists $1\leq i \leq k-1$ such that $\beta_i < \beta_{i+1}$ and $|\beta|_2\neq 0$, the minimal value of $dex$ corresponds to the bundle $(2,1,1;1,0,...,0)$. Eq. (\ref{sympdex}) then says that it cannot appear for $n\geq 5$. Then one only has $(p,q,r;0,...,0)$ or $(p,p,p;\beta_{k+1}, ..., \beta_n)$, and $\beta_{k+1}\neq 0$. As a consequence of eq. (\ref{sympineqdexrank}), there is at least once the bundle $(1, 0, ..., 0)$, therefore $rank({\cal F})-rank((1,0,...,0))=6n-19$, $dex({\cal F})-dex((1,0,...,0))=2n-3$. This last equation gives as the only possibility for the second type bundles that $p=1$, and (for $n\geq 6$) $|\beta|_2=1$. Studying separately $n=5$ and $n\geq 6$ one checks that there are no other cases.

\end{proof}

\begin{proof}[Proof of the classification of Table \ref{tablefoursym}]
As a consequence of the previous proposition, it is sufficient to show that for $k\geq 4$, there is no bundle ${\cal F}$ with the good properties. As one knows $rank({\cal F})$ and $dex({\cal F})$, one finds that, except for the case $k=4, n=5$, one of the summands must be $(1,0,..,0)$ (otherwise $\frac{dex({\cal F})}{rank({\cal F})}\geq \frac{2}{k}$). As there cannot be two such bundles, one can write 
$$
\frac{dex({\cal F})-dex((1,0,...,0))}{rank({\cal F})-rank((1,0,...,0))}\geq \frac{2}{k}\,\,\, ,
$$ 
which gives $n\leq k+\frac{1}{2}+\frac{4}{k}$. This implies that the only cases that have to be studied are: $(k,n)=(4,5),(5,6),(6,7),(7,8),(8,9)$. If $(k,n)=(4,5)$, as for at least one bundle $\beta_5\neq 0$, a similar reasoning on $\frac{rank}{dex}$ tells us that there must be one bundle $(1,0,...,0)$. Then, simple combinatorics prevent any bundle to have the good properties. The remaining cases can be inspected explicitely.
\end{proof}

\subsection{Odd Orthogonal Grassmannians}


The odd orthogonal Grassmannian $OGr(k,2n+1)$ will be thought of as the quotient $G/P_k$, where $P_k$ is the maximal parabolic subgroup containing the standard Borel subgroup of positive roots in $G=SO(2n+1, \mathbb{C})$. Every irreducible homogeneous bundle is represented by its highest weight. A weight is represented by $\beta=(\beta_1,...,\beta_n)$ or by $(\beta_1,...,\beta_k;\beta_{k+1},...,\beta_n)$ when this notation is needed, where $\beta=(\beta_1-\beta_2)\lambda_1+\lambda_2(\beta_2-\beta_3)+...+\lambda_{n-1}(\beta_{n-1}-\beta_n)+2\lambda_n \beta_n$, the $\lambda_i$'s are the fundamental weights for $SO(2n+1, \mathbb{C})$, and the $\beta_i$'s are all integers or all half integers. Notice that the parabolic algebra $Lie(P_k)$ has Levi factor $sl(k)\oplus so(2(n-k)+1)$ in general ($k\neq 1, n$), which is straightforward by looking at the Dynkin diagram.

\begin{notation}
As we work with globally generated bundles, from now on the notation will change: to indicate a bundle with highest weight $\beta$ as before, we will write $(-\beta_k, ... , -\beta_1; \beta_{k+1}, ... , \beta_n)$, which is equivalent to taking the highest weight of the dual representation. In this way, a bundle $\alpha=(\alpha_1,...,\alpha_n)$ (according to the new notation) is globally generated when $\alpha_1\geq ...\geq\alpha_n\geq 0$, i.e. when the weight $\alpha$ is dominant under the action of $G$. To understand why we use this "dualized" notation, refer to the explication before Remark \ref{globgenord}.
\end{notation}

\begin{example}
Over the orthogonal Grassmannian \emph{OGr(3,15)}, the dual tautological bundle ${\cal U}^*$ (of rank $3$) will be denoted again by $(1,0,0;0,0,0,0)$, the tautological "orthogonal" bundle ${\cal U}^{\perp}/{\cal U}$ (of rank $9$) by $(0,0,0;1,0,0,0)$. With ${\cal T}_{+\frac{1}{2}}$ we will denote the bundle coming from the representation $(-\frac{1}{2}, ... , -\frac{1}{2};\frac{1}{2}, ... , \frac{1}{2})$.  
\end{example}

The dimension of a bundle can be calculated explicitely: suppose $k\neq 1, n$;
\[
rank(\beta_1,...,\beta_n)=dim_{sl(k)}(\beta_1,...,\beta_k)\times dim_{so(2(n-k)+1)}(\beta_{k+1},...,\beta_n)
\] where
\[
dim_{so(2r+1)}(\beta_1,...\beta_r) = \prod\limits_{1\leq i<j\leq r} \frac{j-i+\beta_i-\beta_j}{j-i} \prod\limits_{1\leq i\leq j\leq r} \frac{2r+1-j-i+\beta_i+\beta_j}{2r+1-j-i}
\]
is the Weyl character formula (see \cite[Chapter 24, Equation 24.29]{Fulton}). This formula is a consequence of the form of the Levi factor of $P$ and of Remark \ref{remrank}. A similar formula holds when $k=1, n$.

The definition of the function \emph{dex} is the same as before (see Remark \ref{remdex}). Therefore, fourfolds with trivial canonical bundle correspond to homogeneous vector bundles ${\cal F}=\sum_i {\cal E}_i$, with
\begin{equation}
\label{oddrank}
\sum_i rank({\cal E}_i)=k(2n+1-k)-\frac{k(k+1)}{2}-4\,\,\, ,
\end{equation}
\begin{equation}
\label{odddex}
\sum_i dex({\cal E}_i)=2n-k\,\,\, .
\end{equation}

\subsubsection{Classification in odd orthogonal Grassmannian}

We recall that $OGr(k, 2n+1)$ is embedded naturally in $Gr(k, 2n+1)$ as the zero locus of a general section of $S^2 \cU^*$. The following lemma is similar to Lemma \ref{lemmaatleast}.

\begin{lemma}
Suppose ${\cal F}=\oplus_i {\cal E}_i$ is as in the hypothesis of Theorem \ref{thm4class}. In order to find new varieties in OGr(k,2n+1) with respect to the case of the ordinary Grassmannian, it is necessary that for at least one bundle ${\cal E}_i=\beta$, $|\beta|_2 \neq 0$ or the $\beta_i$'s are not integers (they can be half integers).
\end{lemma}

\begin{proof}
The proof is the same as the one for Lemma \ref{lemmaatleast}. Notice only that half integer weights are associated to Spin representations.
\end{proof}

One can assume that $({\cal U}^*)$ over \emph{OGr(k,2n+1)} do not appear as summand in ${\cal F}= \oplus_i {\cal E}_i$;
in fact, taking a zero section locus of this bundle in $OGr(k, 2n+1)$ is equivalent to restricting to the space $OGr(k, 2n)$.

The classification will be made in three steps: one has to distinguish the three particular cases: $k=n$, $k=n-1$, $k\leq n-2$. This is a consequence of the difference in these cases of the Dynkin diagram of $B_n$ with the k-th root removed. The classification in Table \ref{tablefourodd} is a direct consequence of the following propositions.

\begin{proposition}[Classification for $k=n$]
\label{oddk=n}
Over OGr(n,2n+1), as suitable ${\cal F}$ (satisfying the hypothesis of Theorem \ref{thm4class}) one has all and only the cases $(oy)$ appearing in Table \ref{tablefourodd}.
\end{proposition}

\begin{proof}
Under the hypothesis of the proposition, $(\frac{1}{2},...,\frac{1}{2})$ is a line bundle, "squared root" of ${\cal O}(1)$. Moreover, one has $rank({\cal F})=\frac{n(n+1)}{2}-4$, $dex({\cal F})=n$. One can study the rate $\frac{dex}{rank}$ and obtain constraints. In fact $dex(2,0,...,0)=n+1$, so this bundle cannot appear. On the other hand $dex(1,1,0,...,0)=n-1=dex(1,...,1,0)$, and for all the other bundles which are not line bundles, $dex$ is greater than $n$. If $(1,1,0,...,0)$ appears, $rank({\cal F})-rank(1,1,0,...,0)=n-4$. As there must be also at least one bundle with half integers, the only possibility is $n=6$, i.e. $(oy6)$. Therefore, except for this case, one checks that for all the other possible bundles, and therefore for ${\cal F}$, $\frac{dex}{rank}\geq \frac{1}{2}$, which implies $n\leq 4$. This gives the other cases $(oy)$.
\end{proof}

\begin{proposition}[Classification for $k=n-1$]
\label{oddk=n-1}
Over OGr(n-1,2n+1), as suitable ${\cal F}$ (satisfying the hypothesis of Theorem \ref{thm4class}) one has all and only the cases $(ox)$ appearing in Table \ref{tablefourodd}.
\end{proposition}

\begin{proof}
In this case notice that $(0,...,0;\beta)$ is of rank $2\beta+1$. If $n\neq 3, 4$, the minimal ratio $\frac{dex}{rank}$ is $\frac{2}{n-1}$ given by $(1,1,0,...,0)$. But this implies $n\leq 3$. So the only cases to study are $n=3$, $n=4$, and this gives the cases $(ox)$.
\end{proof}

\begin{proposition}[Classification for $k\leq n-2$]
\label{oddk=n-2etmoins}
If $k\leq n-2$, over OGr(k,2n+1), as suitable ${\cal F}$ (satisfying the hypothesis of Theorem \ref{thm4class}) one has all and only the cases $(ob0)$, $(ob0.1)$ appearing in Table \ref{tablefourodd}.
\end{proposition}

\begin{proof}
With this hypothesis too, except for $k\leq 3$, we have $\frac{dex({\cal F})}{rank({\cal F})}\geq \frac{2}{k}$, which implies $k\leq 1$. So three cases have to be considered.

If $k=1$, one easily sees that no possibility matches the requirements.

If $k=2$, $rank({\cal F})=4n-9$, $dex({\cal F})=2n-2$, and if $n\geq 4$ then $\frac{dex({\cal F})}{rank({\cal F})}<1$, therefore there must be at least one bundle $(\frac{1}{2},...,\frac{1}{2})$, whose rank is $2^{n-2}$. This means that $2^{n-2}\leq 4n-9$, so $3\leq n \leq 5$. For $n=4$, one gets the cases $(ob0)$, $(ob0.1)$, and no case for $n=5$.

If $k=3$, the same argument as before gives $5\leq n\leq 8$, and inspecting case by case one finds that no other variety arises.
\end{proof}

\subsection{Even Orthogonal Grassmannians}


The even orthogonal Grassmannian $OGr(k,2n)$ will be thought of as the quotient $G/P_k$, where $P_k$ is the maximal parabolic subgroup containing the standard Borel subgroup of positive roots in $G=SO(2n, \mathbb{C})$. Every irreducible homogeneous bundle is represented by its highest weight. A weight is represented by $\beta=(\beta_1,...,\beta_n)$ or by $(\beta_1,...,\beta_k;\beta_{k+1},...,\beta_n)$ when this notation is needed, where $\beta=(\beta_1-\beta_2)\lambda_1+\lambda_2(\beta_2-\beta_3)+...+\lambda_{n-1}(\beta_{n-1}-\beta_n)+\lambda_n (\beta_{n-1}+\beta_n)$, the $\lambda_i$'s are the fundamental weights for $SO(2n, \mathbb{C})$, and the $\beta_i$'s are all integers or all half integers. Notice that the parabolic algebra $Lie(P_k)$ has Levi factor $sl(k)\oplus so(2(n-k))$ in general ($k\neq 1, n, n-1$), which is straightforward by looking at the Dynkin diagram.

\begin{notation}
As we work with globally generated bundles, from now on the notation will change: to indicate a bundle with highest weight $\beta$ as before, we will write $(-\beta_k, ... , -\beta_1; \beta_{k+1}, ... , \beta_n)$, which is equivalent to taking the highest weight of the dual representation. In this way, a bundle $\alpha=(\alpha_1,...,\alpha_n)$ (according to the new notation) is globally generated when $\alpha_1\geq ...\geq \alpha_{n-1} \geq|\alpha_n|$, i.e. when the weight $\alpha$ is dominant under the action of $G$. To understand why we use this "dualized" notation, refer to the explication before Remark \ref{globgenord}.
\end{notation}

\begin{example}
Over the orthogonal Grassmannian \emph{OGr(3,14)}, with ${\cal T}_{\pm \frac{1}{2}}$ we will denote $(-\frac{1}{2}, ... , -\frac{1}{2};\frac{1}{2}, ... , \frac{1}{2}, \pm \frac{1}{2})$ the bundle coming from the spin representations.
\end{example} 

The dimension of a bundle can be calculated explicitely: suppose $k\neq 1, n, n-1$; then
\[
rank(\beta_1,...,\beta_n)=dim_{sl(k)}(\beta_1,...,\beta_k)\times dim_{so(2(n-k))}(\beta_{k+1},...,\beta_n)
\] where
\[
dim_{so(2r)}(\beta_1,...\beta_r) = \prod\limits_{1\leq i<j\leq r} \frac{j-i+\beta_i-\beta_j}{j-i} \prod\limits_{1\leq i<j\leq r} \frac{2r-j-i+\beta_i+\beta_j}{2r-j-i}
\]
is the Weyl character formula (see \cite[Chapter 24, Equation 24.41]{Fulton}). This formula is a consequence of the form of the Levi factor of $P$ and of Remark \ref{remrank}. A similar formula holds when $k=1, n$.

The definition of the function \emph{dex} is the same as before (see Remark \ref{remdex}). Therefore, fourfolds with trivial canonical bundle correspond to homogeneous vector bundles ${\cal F}=\sum_i {\cal E}_i$, with
\begin{equation}
\label{evenrank}
\sum_i rank({\cal E}_i)=k(2n-k)-\frac{k(k+1)}{2}-4\,\,\, ,
\end{equation}
\begin{equation}
\label{evendex}
\sum_i dex({\cal E}_i)=2n-k-1\,\,\, .
\end{equation}

\subsubsection{Classification in even orthogonal Grassmannian}

We recall that $OGr(k, 2n)$ is embedded naturally in $Gr(k, 2n)$ as the zero locus of a general section of $S^2 \cU^*$. The following lemma is similar to Lemma \ref{lemmaatleast}.

\begin{lemma}
Suppose ${\cal F}=\oplus_i {\cal E}_i$ is as in the hypothesis of Theorem \ref{thm4class}. In order to find new varieties in OGr(k,2n) with respect to the case of the ordinary Grassmannian, it is necessary that for at least one bundle ${\cal E}_i=\beta$, $|\beta|_2 \neq 0$ or the $\beta_i$'s are not integers (they can be half integers).
\end{lemma}

\begin{proof}
The proof is the same as the one for Lemma \ref{lemmaatleast}. Notice only that half integer weights are associated to Spin representations.
\end{proof}

One can assume that ${\cal U}^*$ over \emph{OGr(k,2n)} does not appear as summand in ${\cal F}= \oplus_i {\cal E}_i$;
in fact, taking a zero section locus of this bundle in $OGr(k, 2n)$ is equivalent to restricting to the space $OGr(k, 2n-1)$.

The classification will be made in three steps: one has to distinguish the three particular cases: $k=n$, $k=n-2$, $k\leq n-3$. This is due to the difference in these cases of the Dynkin diagram of $D_n$ with the k-th root removed. The classification of Table \ref{tablefoureven} is a direct consequence of the following propositions.

\begin{proposition}[Classification for $k=n$]
\label{evenk=n}
Over OGr(n,2n), as suitable ${\cal F}$ (satisfying the hypothesis of Theorem \ref{thm4class}) one has all and only the cases $(oz)$ appearing in Table \ref{tablefoureven}.
\end{proposition}

\begin{proof}
The only bundles that do not appear in the classical Grassmannian are those with half integer coefficients. One checks that equations (\ref{evenrank}) and (\ref{evendex}) make it impossible to have twice the bundle $(\frac{1}{2},...,\frac{1}{2}, -\frac{1}{2})$. If it appears even once, eq. (\ref{evendex}) implies that no bundle with integer coefficients which is not a line bundle can appear, therefore $\frac{dex({\cal F})}{rank({\cal F})}\geq \frac{1}{2}$, which gives $n=5, 6$. One obtains therefore $(oz1)$ and $(oz2)$. If the bundle $(\frac{1}{2},...,\frac{1}{2}, -\frac{1}{2})$ is not present, the same argument as before gives that $\frac{dex({\cal F})}{rank({\cal F})}\geq \frac{1}{2}$, which gives the remaining cases $(oz)$.
\end{proof}

\begin{proposition}[Classification for $k=n-2$]
\label{evenk=n-2}
Over OGr(n-2,2n), as suitable ${\cal F}$ (satisfying the hypothesis of Theorem \ref{thm4class}) one has all and only the cases $(ow)$ appearing in Table \ref{tablefoureven}.
\end{proposition}

\begin{proof}
In this case the Levi factor of $Lie(P)$ is given by the Dynkin diagram $D_n$ where the (n-2)-th root has been removed. Therefore its semisimple part is equal to $sl(n-2)\oplus sl(2)\oplus sl(2)$. Then, following Remark \ref{remrank}, we have for a bundle $\beta=(\beta_1,...,\beta_n)$,
\[
rank(\beta)=dim_{sl(n-2)}(\beta_1,...,\beta_{n-2})\times dim_{sl(2)}(\beta_{n-1},\beta_n)\times dim_{sl(2)}(\beta_{n-1},-\beta_n)
\]
If $n\neq 3, 4, 5$, the minimal ratio $\frac{dex}{rank}$ is $\frac{2}{n-2}$ given by $(1,1,0,...,0)$. But this implies $n\leq 4$. So the only cases to study are $n=3$, $n=4$ and $n=5$, and this gives the cases $(ow)$.
\end{proof}

\begin{proposition}[Classification for $k\leq n-3$]
\label{evenk=n-3etmoins}
If $k\leq n-3$, over OGr(k,2n), as suitable ${\cal F}$ (satisfying the hypothesis of Theorem \ref{thm4class}) one has all and only the cases $(ob)$ appearing in Table \ref{tablefoureven}.
\end{proposition}

\begin{proof}
With this hypothesis, except for $k\leq 3$, we have $\frac{dex({\cal F})}{rank({\cal F})}\geq \frac{2}{k}$, which implies $k\leq 1$. So three cases have to be considered.

If $k=1$, one easily sees that no possibility matches the requirements.

If $k=2$, as in the analogous proposition for the odd orthogonal Grassmannian, there must be at least one bundle $(\frac{1}{2},...,\pm \frac{1}{2})$, whose rank is $2^{n-3}$. As $rank({\cal F})=4n-11$, this gives $5\leq n\leq 7$. Therefore, studying case by case, one recovers all the cases $(ob)$.

If $k=3$, the same reason as before gives $6\leq n\leq 9$. But in all these cases, $\frac{dex({\cal F})}{rank({\cal F})}$ is less than $\frac{1}{2}$, therefore no other case arises.
\end{proof}

\subsection{The case of $OGr(n-1,2n)$} 
\label{ssOGr(n-1,2n)}

Let $OGr(n-1,2n)$ be the orthogonal Grassmannian of isotropic $(n-1)$-planes in a $2n$-dimensional complex vector space. This variety is different in nature from those considered up to now, as it is not a generalized Grassmannian in the usual sense (it has Picard rank equal to $2$). 

It is well known that the zero locus of a general section of $S^2{\cal U}^*$ in $Gr(n,2n)$ consists of two components $OGr_+$ and $OGr_-$, each of which is a copy of $OGr(n,2n)$. For each point $W_+\in OGr_+$, and $W_-\in OGr_-$ such that $dim(W_+\cap W_-)=n-1$, let us denote $W:= W_+\cap W_-$. Then $W\in G:= OGr(n-1,2n)$. Moreover, given $W\in G$, one can recover $W_+$ and $W_-$ in a unique way, i.e. there exist two morphisms 
$$
\pi_{+/-}:G\to OGr_{+/-}\,\,\, , \,\,\, W\mapsto W_{+/-} \,\,\, .
$$
 
Over $OGr_+$ (respectively $OGr_-$) there is a line bundle ${\cal O}(\frac{1}{2})_+$ (resp. ${\cal O}(\frac{1}{2})_-$) which is the square root of the restriction to $OGr_+$ (resp. $OGr_-$) of ${\cal O}(1)$ over $Gr(n,2n)$. Therefore there are two line bundles ${\cal L}_+:=\pi_+^*{\cal O}(\frac{1}{2})_+$ and ${\cal L}_-:=\pi_-^*{\cal O}(\frac{1}{2})_-$ over $G$. In fact, it can be shown that the Picard group of $G$ is generated by those two line bundles. This description corresponds to the following picture from the point of view of quotients of $SO(2n, \mathbb{C})$. 

Recall that $P_i$ denotes the parabolic subgroup of $SO(2n, \mathbb{C})$ corresponding to the $i$-th simple root. The even orthogonal Grassmannian $OGr(n-1,2n)$ can be thought of as the quotient $G/P_{n,n-1}$, where $P_{n,n-1}$ is the parabolic subgroup $P_n \cap P_{n-1}$. The reason why $G$ is not considered to be a Grassmannian is exactly because the parabolic subgroup $P_{n,n-1}$ is associated to two, and not one, simple roots. This also explains why $Pic(G)=\mathbb{Z}^2$ (general theory of homogeneous spaces). The two morphisms $\pi_+$ and $\pi_-$ correspond to the two projections 
$$
G=G/P_{n,n-1} \to G/P_{n}\cong OGr_+$$
and 
$$
G=G/P_{n,n-1} \to G/P_{n-1}\cong OGr_-\,\,\, .
$$
Now we want to understand the relation between ${\cal L}_+$, ${\cal L}_-$, and ${\cal O}(1)$ (which is the restriction to $G$ of the Pl\"ucker line bundle ${\cal O}(1)$ over $Gr(n-1,2n)$). Let ${\cal U}$ be the tautological bundle of rank $n-1$ (and ${\cal U}^*$ its dual) on $G$. As already mentioned for the classical Grassmannians, there is a vector bundle ${\cal U}^\perp /{\cal U}$ over $G$ of rank $2$. Let us also denote ${\cal U}_{\pm}$ the tautological bundles over $OGr_{\pm}$ restricted to $G$. Then, ${\cal U}^{\perp}={\cal U}_++{\cal U}_-\subset \mathbb{C}^{2n}$, where the sum is not direct, and, as a consequence, 
$$
{\cal U}^\perp /{\cal U}={\cal U}_+/{\cal U}\oplus {\cal U}_-/{\cal U}\,\,\, .
$$
In this second equation the sum is actually a direct sum (this comes from the fact that $W= W_+\cap W_-$ for $W$ an isotropic $(n-1)$-plane in $\mathbb{C}^{2n}$). The quadratic form on $\mathbb{C}^{2n}$ restricts to a form on ${\cal U}^{\perp}$. Since this form descends to a form on ${\cal U}^{\perp}/{\cal U}$ which is non degenerate, then ${\cal U}^{\perp}/{\cal U}\cong ({\cal U}^{\perp}/{\cal U})^*$, which implies $det({\cal U}^{\perp}/{\cal U})=0$. Moreover $det({\cal U}_{\pm})=\pi_{\pm}^*{\cal O}(\frac{1}{2})_{\pm}^{\otimes 2}$. By taking the determinant of the bundles in the previous equation, we get the important relation
$$
{\cal L}_+\otimes{\cal L}_-={\cal O}(1) \,\,\,.
$$
The following theorem holds:

\begin{theorem}
\label{theorem4}
Let \emph{Y} be a fourfold with $K_Y={\cal O}_Y$ which is the variety of zeroes of a general section of a homogeneous, completely reducible, globally generated vector bundle ${\cal F}$ over the even orthogonal Grassmannian OGr(n-1, 2n) (and which does not appear in the  analogous classification for the classical Grassmannian). Up to identifications, the only possible cases are those appearing in Table \ref{tablefourn-1}.
\end{theorem}

For the proof of the theorem, every irreducible homogeneous bundle is represented by its highest weight. The notations for the bundles on $OGr(n-1,2n)$ are the same as those used for $OGr(k,2n)$ for general $k$, as well as the formula for $rank$ and $dex$ of the bundles (notice that the parabolic algebra $Lie(P_{n-1,n})$ has Levi factor $sl(n-1)$). The line bundles ${\cal L}_+$ and ${\cal L}_-$ will be respectively denoted by $(\frac{1}{2},...,\frac{1}{2})$ and $(\frac{1}{2},...\frac{1}{2},-\frac{1}{2})$. Finally, the canonical bundle of $OGr(n-1,2n)$ is a tensor power of ${\cal O}(1)$.

Fourfolds with trivial canonical bundle correspond to homogeneous vector bundles ${\cal F}=\sum_i {\cal E}_i$, with
\begin{equation}
\label{n-1rank}
\sum_i rank({\cal E}_i)=(n-1)(n+1)-\frac{(n-1)(n)}{2}-4\,\,\, ,
\end{equation}
\begin{equation}
\label{n-1dex}
\sum_i dex({\cal E}_i)=n\,\,\, .
\end{equation}

In order to find new varieties with respect to the case of the ordinary Grassmannian, it is necessary that for at least one bundle ${\cal E}_i=\beta$, $|\beta|_n \neq 0$ or the $\beta_i$'s are not integers (they can be half integers).

\begin{proof}[Proof of theorem \ref{theorem4}]
Suppose the variety is embedded in $OGr(n-1,2n)$, with $n\geq 4$. Then, for any irreducible homogeneous globally generated bundle ${\cal E}_i$ which is a component of ${\cal F}$, we have $\frac{dex({\cal E}_i)}{rank({\cal E}_i)}\geq \frac{2}{n-1}$. By using equations (\ref{n-1dex}) and (\ref{n-1rank}), we get the inequality $n\leq 5$. Then, the theorem follows by inspecting all the possible cases. It should be remarked that, if ${\cal F}=\oplus_i{\cal E}_i=\oplus_i (\beta^i_1,...,\beta^i_n)$ is a suitable vector bundle, then $\sum_i rank({\cal E}_i)*\beta^i_n=0$, which assures that $det({\cal F})$ is a multiple of ${\cal O}(1)$, as $K_{OGr(n-1,2n)}$ is.
\end{proof}

\section{The cases of dimensions two and three}

In this final section, we will give the analogous results of the previous classifications of varieties with trivial canonical bundle in dimensions two and three. The proofs don't present anything new from the previous ones, they follow the exact same strategy, so we omit them. It is perhaps worth remarking that the first problem is to prove the finiteness of the number of such varieties in Grassmannians (for the ordinary Grassmannian, see \cite{Japonais}). The notations are the same used in the previous classifications.

\bigskip
We will begin with the classification in dimension 3.

\begin{theorem}
\label{thm3ord}
Let \emph{Y} be a threefold with $K_Y={\cal O}_Y$ which is the variety of zeroes of a general section of a homogeneous, completely reducible, globally generated vector bundle ${\cal F}$ over Gr(k, n). Up to the identification of Gr(k,n) with Gr(n-k,n), the only possible cases are those appearing in Table \ref{tablethreeord}.
\end{theorem}

This classification, as already mentioned in the introduction, appears also in \cite{Japonais}. However, we point out the fact that in \cite{Japonais}, the cases $(c3)$, $(d3.1)$ and $(d2.1)$ do not appear. The bundles which define them are analogous to the one appearing respectively in the cases $(c2)$, $(d3)$ and $(d2)$, and in fact there are isomorphisms $(c3)\cong (c2)$, $(d3.1)\cong (d3)$ and $(d2.1)\cong (d2)$. These isomorphisms come from the fact that all these cases live in the symplectic Grassmannian $IGr(n,2n)$, on which there is a canonical isomorphism of bundles ${\cal U}^* \cong {\cal Q}$.

\begin{theorem}
\label{thm3class}
Let \emph{Y} be a threefold with $K_Y={\cal O}_Y$ which is the variety of zeroes of a general section of a homogeneous, completely reducible, globally generated vector bundle ${\cal F}$ over the symplectic Grassmannian IGr(k, 2n) (respectively the odd orthogonal Grassmannian OGr(k,2n+1), the even orthogonal Grassmannian OGr(k,2n)) and which does not appear in the  analogous classification for the ordinary Grassmannian. Up to identifications, the only possible cases are those appearing in Table \ref{tablethreesym} (respectively Table \ref{tablethreeodd}, Table \ref{tablethreeeven}).
\end{theorem}

\begin{theorem}
\label{theorem3}
Let \emph{Y} be a threefold with $K_Y={\cal O}_Y$ which is the variety of zeroes of a general section of a homogeneous, completely reducible, globally generated vector bundle ${\cal F}$ over the even orthogonal Grassmannian OGr(n-1, 2n) (and which does not appear in the  analogous classification for the classical Grassmannian). Up to identifications, the only possible cases are those appearing in Table \ref{tablethreen-1}.
\end{theorem}

\bigskip

Now, the classification for dimension two (K3 surfaces). In this case we reported the degree of the surface with respect to the bundle ${\cal O}(\frac{1}{2})$ for the varieties in OGr(n, 2n+1) and OGr(n, 2n), and with respect to the bundle ${\cal O}(1)$ in the other cases. For $OGr(n-1,2n)$, we reported the degree with respect to ${\cal O}(1)$, ${\cal L}_+$ and ${\cal L}_-$. In this way, one also gets the genus of the K3 (by the well known formula $degree=2\, genus-2$).

\begin{theorem}
\label{thm2ord}
Let \emph{Y} be a surface with $K_Y={\cal O}_Y$ which is the variety of zeroes of a general section of a homogeneous, completely reducible, globally generated vector bundle ${\cal F}$ over Gr(k, n). Up to the identification of Gr(k,n) with Gr(n-k,n), the only possible cases are those appearing in Table \ref{tabletwoord}.
\end{theorem}

\begin{theorem}
\label{thm2class}
Let \emph{Y} be a surface with $K_Y={\cal O}_Y$ which is the variety of zeroes of a general section of a homogeneous, completely reducible, globally generated vector bundle ${\cal F}$ over the symplectic Grassmannian IGr(k, 2n) (respectively the odd orthogonal Grassmannian OGr(k,2n+1), the even orthogonal Grassmannian OGr(k,2n)) and which does not appear in the  analogous classification for the ordinary Grassmannian. Up to identifications, the only possible cases are those appearing in Table \ref{tabletwosym} (respectively Table \ref{tabletwoodd}, Table \ref{tabletwoeven}).
\end{theorem}

\begin{theorem}
\label{theorem2}
Let \emph{Y} be a surface with $K_Y={\cal O}_Y$ which is the variety of zeroes of a general section of a homogeneous, completely reducible, globally generated vector bundle ${\cal F}$ over the even orthogonal Grassmannian OGr(n-1, 2n) (and which does not appear in the  analogous classification for the classical Grassmannian). Up to identifications, the only possible cases are those appearing in Table \ref{tabletwon-1}.
\end{theorem}

In the classifications concerning $K3$ surfaces, we have marked the varieties which have already been studied by Mukai with $M.$ In particular: cases $M.(b10)$, $M.(oy5)$, $M.(b13)$, $M.(c6)$ are in \cite{Mukai10}; case $M.(c9)$ is in \cite{Mukai13}; cases $M.(ox5)$, $M.(d3)$ are in \cite{Mukai1820}.

There are many other cases which have not been examined yet, and they are worth being considered in more detail, as we intend to do in the next future.
\smallskip

As a final remark, let us note that the classifications given in this paper give for free the analogous classifications for Fano threefolds, fourfolds and fivefolds. It suffices to individuate the varieties for which the bundle ${\cal F}$ contains a line bundle; this means that the variety with trivial canonical bundle lives in a Fano manifold of dimension greater by one as a hyperplane section of a certain very ample line bundle. What the previous classifications do not give automatically in general is the index of those Fano manifolds.

\appendix

\section{Euler characteristic}
\label{appEulchar}

The computation of the Euler characteristic of the trivial bundle ${\cal O}_Y$ can be done in two different ways. The first one is applicable in general only for the symplectic and odd orthogonal Grassmannians, while the second one can be applied to all the Grassmannians (and actually is lighter in terms of computational time).\newline

For symplectic and odd orthogonal Grassmannians one knows that one can choose multiplicative generators of the cohomology to be the Chern classes of the quotient bundle ${\cal Q}$. This is the tautological quotient bundle of the classical Grassmannian in which the symplectic and odd orthogonal Grassmannians embed naturally (as a reference, one can see \cite[Theorem 8, Theorem 12]{Tamvakis}). For instance, $IGr(k, 2n)$ embeds in $Gr(k, 2n)$ as the zero locus of a section of $\Lambda^2{\cal U}^*$. Then, suppose to be able to express the Chern classes of the bundle ${\cal F}$ which defines $Y$ in $IGr(k, 2n)$ in terms of the Chern classes of ${\cal Q}$, which live in the cohomology of $Gr(k, 2n)$. Then, the computation can be made in this last space. In fact, $[ Y] = c_{top}({\cal F})$ in the cohomology ring of $IGr$, $[ IGr ] =c_{top}(\Lambda^2{\cal U}^*)$ in the cohomology ring of $Gr$, and $\int_{Y}(\cdot )=\int_{IGr}(\cdot ) [Y] =\int_{Gr}(\cdot )[ Y] [ IGr] $. Therefore, one has:
\[
\chi({\cal O}_Y)=\int_Y todd(T_Y)=\int_{IGr} \frac{todd(T_{IGr})}{todd({\cal F})} c_{top}({\cal F})=
\]
\[
=\int_{Gr}\frac{todd(T_{Gr})}{todd({\cal F})}\frac{c_{top}(\Lambda^2({\cal U}^*))}{todd(\Lambda^2({\cal U}^*))} c_{top}({\cal F}).
\] 
In the second equality we have used the fact that the $Todd$ class is multiplicative with respect to short exact sequences, and we have applied this to the normal sequence for $Y$:
\[
0\to T_Y \to T_{IGr}|_Y \to {\cal F}|_Y \to 0 
\]
Then, as in the case of the classification of fourfolds in the classical Grassmannian, one can use the MACAULAY2 package SCHUBERT2 to do the computation. In the symplectic Grassmannian concretely there is essentially one bundle for which one needs to find the relation with the Chern classes of ${\cal Q}$, namely $({\cal U}^{\perp}/{\cal U})(1)$. This is given by the exact sequence:
\[
0\to {\cal U}^{\perp}/{\cal U}(1) \to {\cal Q}(1) \to {\cal U}^*(1) \to 0
\]
For the orthogonal Grassmannian one has to consider also the bundles which correspond to half integer sequences, and in particular the spin bundles. By relating the exterior and symmetric powers of this bundle with the "classical" bundles in the different cases, we can do the computation. For example, for $OGr(n, 2n+1)$, this (line) bundle is just the "square root" of ${\cal O}(1)$, so its first Chern class is half the one of the ample line bundle ${\cal O}(1)$.\newline

\subsection{Cohomology of the even orthogonal Grassmannian}

The method explained earlier cannot be used in general for the even orthogonal Grassmannian, because its cohomology is a little bit more complicated. In this case in fact, the Chern classes of the tautological quotient bundle ${\cal Q}$ do not generate multiplicatively the cohomology ring. One way to proceed is to use a "good" presentation of the cohomology, for which it is relatively easy to understand what the Chern classes of the homogeneous bundles are, and so to compute the integral in the equation mentioned above. The picture we are going to present applies actually, with appropriate modifications, to the other cases too.\newline

\subsubsection{Borel's presentation of the cohomology ring}

We present the multiplicative structure of the cohomology ring of the even orthogonal Grassmannian; it is usually referred to as \emph{Borel's presentation} of the cohomology (as a reference see  \cite{Demazure}, where the more general case of a homogeneous variety is treated). The idea is to inject the cohomology ring into one which is better understood, namely that of a complete flag variety. This has the property that every irreducible homogeneous bundle has rank one. Therefore it will always be possible to split completely a (not necessarily irreducible) bundle and to compute its Chern class as the product of the Chern classes of the line bundles of the splitting.

To be more precise, one considers the projection map $\pi:SO(2n)/B \to SO(2n)/P_k=OGr(k, 2n)$, where $B$ is a Borel subgroup contained in $P_k$. The homogeneous space $SO(2n)/B$ is the complete flag $OF(1,...,n, \CC^{2n})$ of isotropic planes in $\CC^{2n}$. The fiber of $\pi$ over a point $[R]\in OGr(k, 2n)$ is isomorphic to $F(1,..,k, R)\times OF(1,...,n-k, R^{\perp}/R)$ (here the first factor is the usual complete flag in $R$). As $\pi$ is a
fibration, the pull-back morphism $\pi^*$ is an injection. Therefore, after pulling back, one can work in the cohomology of the flag variety $G/B$, where $G=SO(2n)$. 

What one gains, as already anticipated, is the fact that this cohomology is well known and every homogeneous vector bundle can be split as the sum of line bundles. Indeed, let us denote by $X(T)$ the characters of the maximal torus $T$ in $B$. Then one has a morphism (defined in \cite[Section 3]{Demazure}):
\[
c: S_{\mathbb{Q}}[X(T)]\to H^*_{\mathbb{Q}}(G/B)
\]
from the symmetric algebra on the characters with rational coefficients to the rational cohomology of $G/B$. This morphism is surjective, and so identifies a quotient $S_{\mathbb{Q}}[X(T)]/I$ with $H^*_{\mathbb{Q}}(G/B)$. The ideal $I$ can be computed explicitly as the ideal generated by invariant polynomials (without constant terms) under the (natural) action of the Weyl group $W$ of $G$ on $S_{\mathbb{Q}}[X(T)]$. On the other hand, we will need to be able to write explicitly the morphism $c$. One has:
\[
c(f)=\sum_{w\in W|l(w)=deg(f)}\Delta_w(f) X^w
\]
for $f$ homogeneous in $S_{\mathbb{Q}}[X(T)]$, where $X^w$ is the Schubert cohomology class corresponding to the Weyl element $w$ (in the usual \emph{Schubert presentation} of the cohomology ring of homogeneous spaces). Moreover, given a reduced decomposition of $w=s_{i_1}\cdot ... \cdot s_{i_{l(w)}}$ in terms of simple reflections, $\Delta_w=\Delta_{s_{i_1}}\circ ... \circ \Delta_{s_{i_{l(w)}}}$, where
\[
\Delta_{s_i}(f)=\frac{f-s_i(f)}{\alpha_i},
\]
$\alpha_i$ being the i-th simple root. The value of $\Delta_w(f)$ doesn't depend on the chosen reduced decomposition (again, refer to \cite[Theorem 1]{Demazure}).\newline

\subsubsection{Chern class of homogeneous bundles}

Now, having this in mind, the last step to do the computation of the Euler characteristic is to express the Chern classes of a homogeneous vector bundle on $G/P_k$ in $H^*(G/B)$. As already pointed out, a homogeneous completely reducible bundle splits in $H^*(G/B)$ as the sum of line bundles. These line bundles correspond to representations of $B$ (as explained in Section \ref{prelhom}), i.e. to elements of $X(T)$. Fix ${\cal F}$ on $G/P_k$ coming from a representation $V$ of $P_k$. Then one has the weight space decomposition $V=\oplus_{\mu\in X(T)}V_{\mu}^{m_{\mu}}$. As a consequence, 
\[
\pi^*({\cal F})\sim \oplus_{\mu\in X(T)}{\cal L}_{\mu}^{m_{\mu}}
\]
where ${\cal L}_{\mu}$ is the line bundle corresponding to $\mu\in X(T)$. Here, the symbol $\sim$ stands for \emph{"are the same as $T$-homogeneous bundles"}, which implies that they have the same Chern classes. The last ingredient is the fact that the Chern class of ${\cal L}_{\mu}$ is represented inside $S_{\mathbb{Q}}[X(T)]$ by $1+\mu$. As a consequence
\[
\Chern (\pi^*({\cal F}))=\prod_{\mu\in X(T)}(1+\mu)^{m_{\mu}}
\] \newline

Knowing that, we can compute the Chern class of the bundle, products of cohomology classes, integrations, etc. In particular, the integration on $G/P_k$ of a class $f$ of the right degree is given by computing $\Delta_{w_0}(f)$, where $w_0$ is the longest element in $W/W(P_k)$.

\begin{example}
\emph{Here we report the code to use in order to compute the Euler characteristic for the case $(ow6)$, as an example:}\newline

\emph{---Definition of $S_{\mathbb{Q}}[X(T)]$} \newline

 $  S=QQ[a, b, c, d, e]; $ \newline
 
\emph{---Chern class of the tangent bundle and todd class (first terms)} \newline 
 
$ctan=(1+a+b+d)*(1+a+b+e)*(1+a+b-d)*(1+a+b-e)*(1+a+c+d)*(1+a+c+e)*(1+a+c-d)*(1+a+c-e)*(1+c+b+d)*(1+c+b+e)*(1+c+b-d)*(1+c+b-e)*(1+a+b)*(1+a+c)*(1+b+c); $ \newline

 $ ctan1=part({1}, ctan); $ \newline
 
 $ ctan2=part({2}, ctan); $ \newline
 
 $ ctan3=part({3}, ctan); $ \newline
 
 $ ctan4=part({4}, ctan); $ \newline
 
 $tdtan1=ctan1 // 2; $ \newline
 
 $ tdtan2=(ctan1*ctan1+ctan2) // 12; $ \newline
 
 $ tdtan3=(ctan1*ctan2) // 24; $ \newline
 
 $ tdtan4=(-ctan1*ctan1*ctan1*ctan1 +4*ctan1*ctan1*ctan2+3*ctan2*ctan2+ctan1*ctan3-ctan4) // 720; $ \newline
 
\emph{---Chern class of the vector bundle ${\cal F}$ and todd class (first terms)}\newline 
 
$cF=(1+((a+b+c+d+e)//2))*(1+((a+b+c-d-e)//2))*(1+((a+b+c+d+e)//2))*(1+((a+b+c-d-e)//2))*(1+((a+b+c+d+e)//2))*(1+((a+b+c-d-e)//2))*(1+((a+b+c+d+e)//2))*(1+((a+b+c-d-e)//2))*(1+a+b)*(1+b+c)*(1+a+c); $ \newline

 $ cF1=part({1}, cF); $ \newline
 
 $ cF2=part({2}, cF); $ \newline
 
 $ cF3=part({3}, cF); $ \newline
 
 $ cF4=part({4}, cF); $ \newline
 
 $ cF11=part({11}, cF); $ \newline
 
$ tdF1=cF1 // 2; $ \newline

 $ tdF2=(cF1*cF1+cF2) // 12; $ \newline
 
 $ tdF3=(cF1*cF2) // 24; $ \newline
 
 $ tdF4=(-cF1*cF1*cF1*cF1+4*cF1*cF1*cF2+3*cF2*cF2+cF1*cF3-cF4) // 720; $ \newline
 
\emph{---Definition of (the first terms of) $cr=\frac{todd(T_{OG})}{todd({\cal F})}$}\newline
 
$cr1=tdtan1-tdF1; $ \newline

 $ cr2=tdtan2-tdF2-tdF1*cr1; $ \newline
 
 $ cr3=tdtan3-tdF3-tdF1*cr2-tdF2*cr1; $ \newline
 
 $ cr4=tdtan4-tdF4-tdF1*cr3-tdF2*cr2-tdF3*cr1; $ \newline
 
\emph{---Definition of the class $int$ to be integrated}\newline
 
$int=cr4*cF11; $ \newline

\emph{--- Computation of $\Delta_{w_0}(int)$ }\newline

 $ intfifteen=(int-sub(int, {d=>c, c=>d})) // (c-d); $ \newline
 
 $ intfourteen=(intfifteen-sub(intfifteen, {c=>b, b=>c})) // (b-c); $ \newline
 
 $ intthirteen=(intfourteen-sub(intfourteen, {e=>d, d=>e})) // (d-e); $ \newline
 
 $ inttwelve=(intthirteen-sub(intthirteen, {b=>c, c=>b})) // (b-c); $ \newline
 
 $ inteleven=(inttwelve-sub(inttwelve, {d=>-e, e=>-d})) // (d+e); $ \newline
 
 $ intten=(inteleven-sub(inteleven, {d=>c, c=>d})) // (c-d); $ \newline
 
 $ intnine=(intten-sub(intten, {d=>e, e=>d})) // (d-e); $ \newline
 
 $ inteight=(intnine-sub(intnine, {b=>c, c=>b})) // (b-c); $ \newline
 
 $ intseven=(inteight-sub(inteight, {d=>c, c=>d})) // (c-d); $ \newline
 
 $ intsix=(intseven-sub(intseven, {a=>b, b=>a})) // (a-b); $ \newline
 
 $ intfive=(intsix-sub(intsix, {c=>b, b=>c})) // (b-c); $ \newline
 
 $ intfour=(intfive-sub(intfive, {c=>d, d=>c})) // (c-d); $ \newline
 
 $ intthree=(intfour-sub(intfour, {e=>d, d=>e})) // (d-e); $ \newline
 
 $ inttwo=(intthree-sub(intthree, {e=>-d, d=>-e})) // (d+e); $ \newline
 
 $ intone=(inttwo-sub(inttwo, {c=>d, d=>c})) // (c-d) $ \newline
 
 \end{example}

\begin{remark}[The case $(ob5)$, Table \ref{tablefoureven}]
It is the only case in which the Euler characteristic is not equal to $2$, but to $4$. In order to understand better why this happens, we studied in more detail the cohomology of ${\cal O}_Y$ by using the Koszul complex associated to the bundle ${\cal F}$. The method is standard (see Section \ref{secCYvsIHS}), the only difficulty in this case is to express the bundle $\Lambda^k{\cal F}^*$ as a sum of irreducible homogeneous bundles, but this can be done using the program LiE (\cite{Lie}). What one finds is that the variety $(ob5)$ is not connected, and actually it consists of two connected components, which are (therefore) Calabi-Yau varieties. Two questions which can be asked is whether these two components are isomorphic, and if there is a more geometric explanation for the existence of these two components (as is the case for the variety of zeroes of $S^2{\cal Q} $ in $Gr(m,2m)$).
\end{remark}

\section{Tables}
\label{apptables}

In this appendix we report all tables of the varieties we have found, as indicated in the classification theorems in the paper. The labelling of the varieties follows essentially the subdivision of the classification's proofs in lemmas and propositions. In the following we explicit some rules we adopted. 

For the ordinary (respectively symplectic, orthogonal) Grassmannians $Gr(k,n)$ ($IGr(k,m)$, $OGr(k,m)$) with $k\leq n$, if $k=2$ then the varieties are labelled by the letter $(b)$ (resp. $(sb)$, $(ob)$), if $k=3$ they are labelled by $(c)$ (resp. $(sc)$, $(oc)$), and if $k\geq 4$, labels start with the letter $(d)$ (resp. $(sd)$, $(od)$).

Some special cases are to be taken into account.
For the odd orthogonal Grassmannians, varieties inside $OGr(k-2,2k+1)$ are labelled by $(ox)$ and varieties inside $OGr(k,2k+1)$ by $(oy)$.
For the even orthogonal Grassmannians, varieties inside $OGr(k-2,2k)$ are labelled by $(ow)$ and varieties inside $OGr(k,2k)$ by $(oz)$.

Finally, varieties inside $OGr(k-1,2k)$ are labelled by the letter $(oe)$.

\begin{table}
\small
\begin{tabular}{llll}
   case & bundle ${\cal F}$ & Gr(k,n) & $\chi({\cal O}_Y)$ \\
   \hline
   (a) & complete intersection of hypersurfaces & $\mathbb{P}^n$ & \\
   \hline
   (b1) & ${\cal O}(1)\oplus{\cal O}(4)$ & Gr(2,5) & 2 \\
   (b2) & ${\cal O}(2)\oplus{\cal O}(3)$ & Gr(2,5) & 2 \\
   (b2.1) & ${\cal U}^*(2)$ & Gr(2,5) & 2  \smallskip \\
   
   (b4) & $S^2{\cal U}^*\oplus{\cal O}(3)$ & Gr(2,6) & 2 \\
   (b5) & ${\cal U}^*(1)\oplus{\cal O}(2)\oplus{\cal O}(1)$ & Gr(2,6) & 2 \\
   (b6) & ${\cal O}(1)^{\oplus 2}\oplus{\cal O}(2)^{\oplus 2}$ & Gr(2,6) & 2 \\  
   (b6.1) & ${\cal O}(1)^{\oplus 3}\oplus{\cal O}(3)$ & Gr(2,6) & 2 \\
   (b6.2) & ${\cal U}^*(1)^{\oplus 2}$ & Gr(2,6) & 2 \\
   (b12) & $S^3{\cal U}^*$ & Gr(2,6) & 3 \smallskip \\
   
   (b3) & ${\cal Q}(1)\oplus {\cal O}(1)$ & Gr(2,7) & 2 \\
   (b7) & ${\cal O}(1)^{\oplus 5}\oplus {\cal O}(2)$ & Gr(2,7) & 2 \\
   (b8) & $S^2{\cal U}^*\oplus{\cal O}(1)^{\oplus 2}\oplus {\cal O}(2)$ & Gr(2,7) & 2 \\
   (b8.1) & ${\cal U}^*(1)\oplus{\cal O}(1)^{\oplus 4}$ & Gr(2,7) & 2 \\
   (b9) & $S^2{\cal U}^*\oplus{\cal U}^*(1)\oplus{\cal O}(1)$ & Gr(2,7) & 2 \\
   (b10) & $\Lambda^4{\cal Q}\oplus{\cal O}(3)$ & Gr(2,7) & 2 \smallskip \\
   
   (b10.1) & $\Lambda^5{\cal Q}\oplus{\cal O}(1)\oplus{\cal O}(2)$ & Gr(2,8) & 2 \\
   (b10.2) & $\Lambda^5{\cal Q}\oplus{\cal U}^*(1)$ & Gr(2,8) & 2 \\
   (b11) & ${\cal O}(1)^{\oplus 8}$ & Gr(2,8) & 2 \\
   (b13) & $S^2({\cal U}^*)^{\oplus 2} \oplus {\cal O}(1)^{\oplus 2}$ & Gr(2,8) & 2 \\
   (b14) & $S^2({\cal U}^*) \oplus {\cal O}(1)^{\oplus 5}$ & Gr(2,8) & 2 \smallskip \\
   
   (b10.3) & $\Lambda^6{\cal Q}\oplus S^2{\cal U}^*$ & Gr(2,9) & 2 \\
   (b10.4) & $\Lambda^6{\cal Q}\oplus{\cal O}(1)^{\oplus 3}$ & Gr(2,9) & 2 \\
   \hline
   
   (c1) & ${\cal O}(1)^{\oplus 4} \oplus {\cal O}(2)$ & Gr(3,6) & 2 \\
   (c1.1) & ${\cal Q}(1) \oplus {\cal O}(1)^{\oplus 2}$ & Gr(3,6) & 2 \\
   (c2) & $\Lambda^2{\cal Q} \oplus {\cal O}(1)\oplus {\cal O}(3)$ & Gr(3,6) & 2 \\
   (c2.1) & $\Lambda^2{\cal Q} \oplus {\cal O}(2)^{\oplus 2}$ & Gr(3,6) & 2 \smallskip \\
   
   (c4) & $S^2{\cal U}^* \oplus {\cal O}(1) \oplus {\cal O}(2)$ & Gr(3,7) & 2 \\
   (c5) & $\Lambda^3{\cal Q} \oplus {\cal O}(1)^{\oplus 4}$ & Gr(3,7) & 2 \\
   (c5.1) & $\Lambda^3{\cal Q} \oplus \Lambda^2{\cal U}^* \oplus {\cal O}(2)$ & Gr(3,7) & 2 \\
   (c6) & $(\Lambda^2{\cal U}^*)^{\oplus 2} \oplus {\cal O}(1) \oplus {\cal O}(2)$ & Gr(3,7) & 2 \\
   (c6.1) & $\Lambda^2{\cal U}^* \oplus {\cal O}(1)^{\oplus 5}$ & Gr(3,7) & 2 \smallskip \\
   
   (c7) & $\Lambda^3{\cal Q} \oplus {\cal O}(2)$ & Gr(3,8) & 2 \\
   (c7.1) & $\Lambda^4{\cal Q} \oplus S^2{\cal U}^*$ & Gr(3,8) & 2 \\
   (c7.2) & $\Lambda^4{\cal Q} \oplus (\Lambda^2{\cal U}^*)^{\oplus 2}$ & Gr(3,8) & 2 \\
   (c7.3) & $S^2{\cal U}^* \oplus \Lambda^2{\cal U}^* \oplus {\cal O}(1)^{\oplus 2}$ & Gr(3,8) & 2 \\
   (c7.4) & $(\Lambda^2{\cal U}^*)^{\oplus 3} \oplus {\cal O}(1)^{\oplus 2}$ & Gr(3,8) & 2 \\
   \hline
   
   (d6) & $S^2{\cal U}^* \oplus {\cal O}(1) \oplus {\cal O}(2)$ & Gr(4,8) & 4 \\
   (d9) & $\Lambda^2({\cal U}^*) \oplus \Lambda^3({\cal U}^*) \oplus {\cal O}(1)^{\oplus 2}$ & Gr(4,8) & 2 \\
   (d9.1) & $\Lambda^3{\cal Q} \oplus \Lambda^2{\cal U}^* \oplus {\cal O}(1)^{\oplus 2}$ & Gr(4,8) & 2 \smallskip \\
   
   (d1) & $\Lambda^3{\cal U}^* \oplus (\Lambda^2{\cal U}^*)^{\oplus 2}$ & Gr(4,9) & 2 \\
   (d8) & $\Lambda^2{\cal U}^* \oplus \Lambda^3{\cal Q}$ & Gr(4,9) & 2 \smallskip \\
   
   (d5) & $(S^2{\cal U}^*)^{\oplus 2} $ & Gr(4,10) & 0 \\
   (d7) & $\Lambda^3{\cal Q} $ & Gr(4,10) & 3 \smallskip \\
   
   (d2) & $(\Lambda^2{\cal U}^*)^{\oplus 2} \oplus {\cal O}(2)$ & Gr(5,10) & 2 \\
   (d2.1) & $\Lambda^2{\cal Q} \oplus \Lambda^2{\cal U}^* \oplus {\cal O}(2)$ & Gr(5,10) & 2 \smallskip \\
   
   (d4) & $\Lambda^2{\cal U}^* \oplus S^2{\cal U}^* \oplus {\cal O}(1)$ & Gr(5,11) & 2 \smallskip \\
   
   (d3) & $(\Lambda^2{\cal U}^*)^{\oplus 2} \oplus {\cal O}(1)^{\oplus 2}$ & Gr(6,12) & 2 \\
   (d3.1) & $\Lambda^2{\cal Q} \oplus \Lambda^2{\cal U}^* \oplus {\cal O}(1)^{\oplus 2}$ & Gr(6,12) & 2 \\
   
\end{tabular}
\caption{Fourfolds in ordinary Grassmannians, see Theorem \ref{thm4ordin}}
\label{tablefourord}
\end{table}

\begin{table}
\small
\begin{tabular}{llll}
   case & bundle ${\cal F}$ & IGr(k,2n) & $\chi({\cal O}_Y)$ \\
   \hline
   (sb0) & $({\cal U}^{\perp}/ {\cal U})(1)\oplus{\cal O}(3)$ & IGr(2,6) & 2 \\
   (sb0.1) & $({\cal U}^{\perp}/ {\cal U})(2)\oplus{\cal O}(1)$ & IGr(2,6) & 2 \smallskip \\
   
   (sb1) & $ \Lambda^2({\cal U}^{\perp}/ {\cal U})(1)\oplus {\cal O}(1)^{\oplus 2}$ & IGr(2,8) & 2 \\
   (sb2) & $({\cal U}^{\perp}/ {\cal U})(1)\oplus{\cal O}(1)^{\oplus 3}$ & IGr(2,8) & 2 \\
   (sb2.1) & $({\cal U}^{\perp}/ {\cal U})(1)\oplus{\cal U}^*\oplus{\cal O}(2)$ & IGr(2,8) & 2 \\
   (sb3) & $({\cal U}^{\perp}/ {\cal U})(1)\oplus S^2{\cal U}^*$ & IGr(2,8) & 2 \\
   \hline
   
   (sc0.1) & $({\cal U}^{\perp}/ {\cal U})(1)\oplus (\Lambda^2{\cal U}^*)^{\oplus 2}$ & IGr(3,8) & 2 \\
   (sc0.2) & $({\cal U}^{\perp}/ {\cal U})(1)\oplus S^2{\cal U}^*$ & IGr(3,8) & 2 \\
   (sc0.3) & $({\cal U}^{\perp}/ {\cal U})(1)\oplus {\cal U}^* \oplus {\cal O}(1)^{\oplus 3}$ & IGr(3,8) & 2 \\
   (sc0.4) & $({\cal U}^{\perp}/ {\cal U})(1)^{\oplus 2}\oplus {\cal U}^* \oplus {\cal O}(1)$ & IGr(3,8) & 2 \\
   
\end{tabular}
\caption{Fourfolds in symplectic Grassmannians, see Theorem \ref{thm4class}}
\label{tablefoursym}
\end{table}

\begin{table}
\small
\begin{tabular}{llll}
   case & bundle ${\cal F}$ & OGr(k,n) & $\chi({\cal O}_Y)$ \\
   \hline
   (ob0) & ${\cal T}_{+\frac{1}{2}}(1)\oplus {\cal O}(1)^{\oplus 2}\oplus{\cal O}(2)$ & OGr(2,9) & 2 \\
   (ob0.1) & ${\cal T}_{+\frac{1}{2}}(1) \oplus {\cal U}^*(1)\oplus{\cal O}(1)$ & OGr(2,9) & 2 \\
   \hline
   
   (ox1) & ${\cal T}_{+\frac{1}{2}}(1) \oplus {\cal O}(3)$ & OGr(2,7) & 2 \\
   (ox2) & ${\cal T}_{+\frac{1}{2}}(2) \oplus {\cal O}(1)$ & OGr(2,7) & 2  \smallskip \\
   
   (ox3) & ${\cal T}_{+\frac{1}{2}}(1)^{\oplus 3} \oplus {\cal O}(1)^{\oplus 2}$ & OGr(3,9) & 2 \\
   (ox4) & ${\cal T}_{+\frac{1}{2}}(1)^{\oplus 2} \oplus \Lambda^2{\cal U}^* \oplus {\cal O}(1)$ & OGr(3,9) & 2 \\
   (ox5) & ${\cal T}_{+\frac{1}{2}}(1) \oplus  (\Lambda^2{\cal U}^*)^{\oplus 2}$ & OGr(3,9) & 2 \\
   (ox6) & ${\cal T}_{+\frac{1}{2}}(1) \oplus S^2{\cal U}^*$ & OGr(3,9) & 2 \\
   \hline
   
   (oy1) & ${\cal T}_{+\frac{1}{2}}(1) \oplus {\cal T}_{+\frac{1}{2}}(3)$ & OGr(3,7) & 2 \\
   (oy1.1) & ${\cal T}_{+\frac{1}{2}}(2)^{\oplus 2}$ & OGr(3,7) & 2 \smallskip \\
   
   (oy2) & ${\cal T}_{+\frac{1}{2}}(1)^{\oplus 2} \oplus \Lambda^3{\cal U}^*$ & OGr(4,9) & 2 \\
   (oy3) & ${\cal T}_{+\frac{1}{2}}(1)^{\oplus 2} \oplus ({\cal T}_{+\frac{1}{2}}(1) \otimes {\cal U}^*)$ & OGr(4,9) & 2 \\
   (oy4) & ${\cal T}_{+\frac{1}{2}}(1)^{\oplus 5} \oplus {\cal T}_{+\frac{1}{2}}(2)$ & OGr(4,9) & 2 \\
   (oy5) & ${\cal T}_{+\frac{1}{2}}(1)^{\oplus 4} \oplus {\cal O}(1)^{\oplus 2} $ & OGr(4,9) & 2 \smallskip \\
   
   (oy6) & ${\cal T}_{+\frac{1}{2}}(1)^{\oplus 2} \oplus \Lambda^2({\cal U}^*) $ & OGr(6,13) & 2 \\
   
\end{tabular}
\caption{Fourfolds in odd orthogonal Grassmannians, see Theorem \ref{thm4class}}
\label{tablefourodd}
\end{table}

\begin{table}
\small
\begin{tabular}{llll}
   case & bundle ${\cal F}$ & OGr(k,n) & $\chi({\cal O}_Y)$ \\
   \hline
   (ob1) & ${\cal T}_{+\frac{1}{2}}(1)^{\oplus 2} \oplus{\cal O}(3)$ & OGr(2,10) & 2 \\
   (ob1.2) & ${\cal T}_{+\frac{1}{2}}(1) \oplus{\cal T}_{-\frac{1}{2}}(1) \oplus{\cal O}(3)$ & OGr(2,10) & 2 \\
   (ob2) & ${\cal T}_{+\frac{1}{2}}(1) \oplus{\cal O}(1)^{\oplus 5}$ & OGr(2,10) & 2 \\
   (ob3) & ${\cal T}_{+\frac{1}{2}}(1) \oplus S^2({\cal U}^*) \oplus{\cal O}(1)^{\oplus 2}$ & OGr(2,10) & 2 \smallskip \\
   
   (ob4) & ${\cal T}_{+\frac{1}{2}}(1) \oplus{\cal O}(1)^{\oplus 5}$ & OGr(2,12) & 2  \\
   (ob4.1) & ${\cal T}_{+\frac{1}{2}}(1) \oplus S^2({\cal U}^*) \oplus{\cal O}(1)^{\oplus 2}$ & OGr(2,12) & 2  \smallskip \\
   
   (ob5) & ${\cal T}_{+\frac{1}{2}}(1) \oplus{\cal O}(3)$ & OGr(2,14) & 4 \\
   \hline
   
   (ow1) & ${\cal T}_{+\frac{1}{2}}(1)^{\oplus 2} \oplus{\cal O}(3)$ & OGr(2,8) & 2 \\
   (ow1.2) & ${\cal T}_{+\frac{1}{2}}(1) \oplus{\cal T}_{-\frac{1}{2}}(1) \oplus {\cal O}(3)$ & OGr(2,8) & 2 \\
   (ow2) & $({\cal U}^{\perp}/ {\cal U})(1) \oplus {\cal O}(1)$ & OGr(2,8) & 2 \\
   (ow3) & $({\cal T}_{+\frac{1}{2}}(1)\otimes {\cal U}^*) \oplus{\cal O}(1)$ & OGr(2,8) & 2 \\
   (ow4) & $(1,1;1;1) \oplus {\cal O}(1)^{\oplus 2}$ & OGr(2,8) & 2 \\
   (ow5) & ${\cal T}_{+\frac{1}{2}}(1) \oplus{\cal O}(1)^{\oplus 2} \oplus {\cal O}(2)$ & OGr(2,8) & 2 \\
   (ow9) & ${\cal T}_{+\frac{1}{2}}(2) \oplus {\cal T}_{+\frac{1}{2}}(1) \oplus {\cal O}(1)$ & OGr(2,8) & 2 \\
   (ow10) & ${\cal T}_{-\frac{1}{2}}(2) \oplus {\cal T}_{+\frac{1}{2}}(1) \oplus {\cal O}(1)$ & OGr(2,8) & 2 \\
   (ow11) & ${\cal U}^*(1) \oplus {\cal T}_{+\frac{1}{2}}(1) \oplus {\cal O}(1)$ & OGr(2,8) & 2 \smallskip \\
   
   (ow6) & ${\cal T}_{+\frac{1}{2}}(1)^{\oplus 4} \oplus \Lambda^2({\cal U}^*)$ & OGr(3,10) & 2 \\
   (ow7) & ${\cal T}_{+\frac{1}{2}}(1)^{\oplus 2} \oplus {\cal T}_{-\frac{1}{2}}(1)^{\oplus 2} \oplus \Lambda^2({\cal U}^*)$ & OGr(3,10) & 2 \\
   (ow8) & ${\cal T}_{+\frac{1}{2}}(1)^{\oplus 3} \oplus {\cal T}_{-\frac{1}{2}}(1) \oplus \Lambda^2({\cal U}^*)$ & OGr(3,10) & 2 \\
   \hline
   
   (oz3) & ${\cal O}(\frac{1}{2}) \oplus{\cal O}(\frac{5}{2})$ & OGr(4,8) & 2 \\
   (oz7) & ${\cal O}(\frac{3}{2}) \oplus{\cal O}(\frac{3}{2})$ & OGr(4,8) & 2 \smallskip \\
   
   (oz1) & ${\cal U}^*(-\frac{1}{2}) \oplus{\cal O}(\frac{5}{2})$ & OGr(5,10) & 2 \\
   (oz4) & ${\cal O}(\frac{1}{2})^{\oplus 4} \oplus{\cal O}(1)^{\oplus 2}$ & OGr(5,10) & 2 \\
   (oz5) & ${\cal O}(\frac{1}{2})^{\oplus 5} \oplus{\cal O}(\frac{3}{2})$ & OGr(5,10) & 2 \\
   (oz6) & ${\cal O}(\frac{1}{2}) \oplus {\cal U}^*(\frac{1}{2})$ & OGr(5,10) & 2 \smallskip \\
   
   (oz2) & ${\cal U}^*(-\frac{1}{2}) \oplus {\cal O}(\frac{1}{2})^{\oplus 4} \oplus{\cal O}(1)$ & OGr(6,12) & 2 \\
   
\end{tabular}
\caption{Fourfolds in even orthogonal Grassmannians, see Theorem \ref{thm4class}}
\label{tablefoureven}
\end{table}

\begin{table}
\small
\begin{tabular}{llll}
   case & bundle ${\cal F}$ & OGr(n-1,2n) & $\chi({\cal O}_Y)$ \\
   \hline
   (oe1) & $({\cal U}^*\otimes {\cal L}_+)\oplus {\cal L}_- \oplus {\cal L}_-^{\otimes 2}$ & OGr(3,8) & 2 \\
   (oe2) & $({\cal L}_+^{\otimes 2})^{\oplus 2} \oplus {\cal L}_-^{\otimes 2} \oplus {\cal L}_-^{\oplus 2}$ & OGr(3,8) & 2 \\
   (oe3) & $\Lambda^2{\cal U}^*\oplus {\cal L}_-^{\otimes 2} \oplus {\cal L}_+^{\otimes 2}$ & OGr(3,8) & 2 \\
   (oe4) & $\Lambda^2{\cal U}^*\oplus {\cal L}_-(1) \oplus {\cal L}_+$ & OGr(3,8) & 2 \\
   (oe5) & ${\cal O}(1)^{\oplus 3} \oplus {\cal L}_- \oplus {\cal L}_+$ & OGr(3,8) & 2 \\
   (oe6) & ${\cal O}(1)\oplus {\cal L}_-^{\otimes 2} \oplus {\cal L}_+^{\otimes 2} \oplus {\cal L}_- \oplus {\cal L}_+$ & OGr(3,8) & 2 \\
   (oe7) & ${\cal O}(1)^{\oplus 2} \oplus {\cal L}_-^{\oplus 2} \oplus {\cal L}_+^{\otimes 2}$ & OGr(3,8) & 2 \\
   (oe8) & ${\cal O}(2)\oplus {\cal L}_-^{\oplus 2} \oplus {\cal L}_+^{\oplus 2}$ & OGr(3,8) & 2 \\
   (oe9) & ${\cal L}_+^{\otimes 2}(1)\oplus {\cal L}_-^{\oplus 3} \oplus {\cal L}_+$ & OGr(3,8) & 2 \\
   (oe10) & ${\cal L}_-^{\oplus 4} \oplus {\cal L}_+^{\otimes 4}$ & OGr(3,8) & 2 \\
   (oe11) & ${\cal O}(1)\oplus {\cal L}_+(1)\oplus {\cal L}_-^{\oplus 2} \oplus {\cal L}_+$ & OGr(3,8) & 2 \\
   (oe12) & ${\cal L}_+^{\otimes 2}\oplus {\cal L}_-^{\oplus 3} \oplus {\cal L}_+(1)$ & OGr(3,8) & 2 \\
   (oe13) & ${\cal L}_-(1) \oplus {\cal L}_+^{\otimes 2}\oplus {\cal L}_-^{\oplus 2} \oplus {\cal L}_+$ & OGr(3,8) & 2 \\
   (oe14) & ${\cal O}(1)\oplus {\cal L}_-^{\oplus 3} \oplus {\cal L}_+^{\otimes 3}$ & OGr(3,8) & 2 \\
   (oe15) & ${\cal L}_-^{\otimes 2} \oplus {\cal L}_+^{\otimes 3}\oplus {\cal L}_-^{\oplus 2} \oplus {\cal L}_+$ & OGr(3,8) & 2 \smallskip \\
   
   (oe16) & ${\cal L}_-^{\oplus 5} \oplus {\cal L}_+^{\oplus 5}$ & OGr(4,10) & 2 \\
   (oe17) & $\Lambda^2{\cal U}^*\oplus {\cal L}_-^{\oplus 2} \oplus {\cal L}_+^{\oplus 2}$ & OGr(4,10) & 2 \\
   
\end{tabular}
\caption{Fourfolds in OGr(n-1,2n), see Theorem \ref{theorem4}}
\label{tablefourn-1}
\end{table}

\begin{table}
\small
\begin{tabular}{lll}
   case & bundle ${\cal F}$ & Gr(k,n) \\
   \hline
   (a) & complete intersection of hypersurfaces & $\mathbb{P}^n$ \\
   \hline
   (b1) & ${\cal O}(4)$ & Gr(2,4)  \smallskip \\
   
   (b2) & ${\cal O}(1)^{\oplus 2}\oplus{\cal O}(3)$ & Gr(2,5) \\
   (b3) & ${\cal O}(2)^{\oplus 2}\oplus{\cal O}(1)$ & Gr(2,5) \\
   (b4) & $\Lambda^2{\cal Q}(1)$ & Gr(2,5) \\
   (b5) & ${\cal U}^*(1)\oplus {\cal O}(2)$ & Gr(2,5)  \smallskip \\
   
   (b6) & ${\cal Q}(1)\oplus{\cal O}(1)$ & Gr(2,6) \\
   (b9) & $\Lambda^3{\cal Q}\oplus{\cal O}(3)$ & Gr(2,6) \\
   (b12) & $S^2{\cal U}^*\oplus{\cal O}(1)\oplus{\cal O}(2)$ & Gr(2,6) \\  
   (b13) & ${\cal O}(1)^{\oplus 4}\oplus{\cal O}(2)$ & Gr(2,6) \\
   (b14) & $S^2{\cal U}^*\oplus{\cal U}^*(1)$ & Gr(2,6) \\
   (b15) & ${\cal U}^*(1)\oplus{\cal O}(1)^{\oplus 3}$ & Gr(2,6) \smallskip \\
   
   (b8) & $\Lambda^4 {\cal Q}\oplus {\cal U}^*(1)$ & Gr(2,7) \\
   (b10) & $\Lambda^4 {\cal Q}\oplus {\cal O}(2)\oplus{\cal O}(1)$ & Gr(2,7) \\
   (b16) & $(S^2{\cal U}^*)^{\oplus 2}\oplus{\cal O}(1)$ & Gr(2,7) \\
   (b17) & $S^2{\cal U}^*\oplus{\cal O}(1)^{\oplus 4}$ & Gr(2,7) \\
   (b18) & ${\cal O}(1)^{\oplus 7}$ & Gr(2,7) \smallskip \\
   
   (b7) & $\Lambda^5{\cal Q}\oplus S^2{\cal U}^*$ & Gr(2,8) \\
   (b11) & $\Lambda^5{\cal Q}\oplus {\cal O}(1)^{\oplus 3}$ & Gr(2,8) \\
   \hline
   
   (c2) & ${\cal Q}(1) \oplus \Lambda^2{\cal Q}$ & Gr(3,6) \\
   (c3) & ${\cal Q}(1) \oplus \Lambda^2{\cal U}^*$ & Gr(3,6) \\
   (c4) & $\Lambda^2{\cal Q} \oplus {\cal O}(1)^{\oplus 2}\oplus {\cal O}(2)$ & Gr(3,6) \\
   (c8) & ${\cal O}(1)^{\oplus 6}$ & Gr(3,6) \smallskip \\
   
   (c6) & $\Lambda^3{\cal Q}^{\oplus 2} \oplus {\cal O}(1)$ & Gr(3,7) \\
   (c7) & $\Lambda^3{\cal Q}\oplus\Lambda^2{\cal U}^* \oplus {\cal O}(1)^{\oplus 2}$ & Gr(3,7) \\
   (c9) & ${\cal O}(1)^{\oplus 3} \oplus S^2{\cal U}^* $ & Gr(3,7) \\
   (c10) & $\Lambda^2{\cal U}^* \oplus {\cal O}(1)^{\oplus 3}$ & Gr(3,7) \smallskip \\
   
   (c1) & $\Lambda^3{\cal Q} \oplus {\cal O}(1)^{\oplus 2}$ & Gr(3,8) \\
   (c11) & $(S^2{\cal U}^*)^{\oplus 2}$ & Gr(3,8) \\
   (c12) & $S^2{\cal U}^* \oplus (\Lambda^2{\cal U}^*)^{\oplus 2}$ & Gr(3,8) \\
   (c13) & $(\Lambda^2{\cal U}^*)^{\oplus 4}$ & Gr(3,8) \\
   \hline
   
   (d3) & $(\Lambda^2{\cal U}^*)^{\oplus 2} \oplus {\cal O}(2)$ & Gr(4,8) \\
   (d3.1) & $\Lambda^2{\cal U}^* \oplus \Lambda^2({\cal Q}) \oplus {\cal O}(2)$ & Gr(4,8) \\
   (d4) & $S^2{\cal U}^* \oplus {\cal O}(1)^{\oplus 3}$ & Gr(4,8) \smallskip \\
   
   (d1) & $\Lambda^2{\cal U}^* \oplus S^2{\cal U}^*\oplus {\cal O}(1)$ & Gr(4,9) \smallskip \\
   
   (d2) & $(\Lambda^2{\cal U}^*)^{\oplus 2}\oplus {\cal O}(1)^{\oplus 2}$ & Gr(5,10) \\
   (d2.1) & $\Lambda^2{\cal U}^*\oplus \Lambda^2{\cal Q} \oplus {\cal O}(1)^{\oplus 2}$ & Gr(5,10) \\
   
\end{tabular}
\caption{Threefolds in ordinary Grassmannians, see Theorem \ref{thm3ord}}
\label{tablethreeord}
\end{table}

\begin{table}
\small
\begin{tabular}{lll}
   case & bundle ${\cal F}$ & IGr(k,2n) \\
   \hline
   (sb1) & $({\cal U}^{\perp}/ {\cal U})(2)\oplus{\cal U}^*$ & IGr(2,6) \\
   (sb2) & $({\cal U}^{\perp}/ {\cal U})(1)\oplus{\cal O}(1)\oplus{\cal O}(2)$ & IGr(2,6) \\
   (sb3) & $({\cal U}^{\perp}/ {\cal U})(1)\oplus{\cal U}^*(1)$ & IGr(2,6) \smallskip \\
   
   (sb4) & $ \Lambda^2({\cal U}^{\perp}/ {\cal U})(1)\oplus {\cal O}(1)\oplus{\cal U}^*$ & IGr(2,8) \\
   (sb5) & $({\cal U}^{\perp}/ {\cal U})(1)\oplus {\cal U}^*\oplus{\cal O}(1)^{\oplus 2}$ & IGr(2,8) \\
   \hline
   
   (sc1) & $({\cal U}^{\perp}/ {\cal U})(1)\oplus \Lambda^2{\cal U}^* \oplus{\cal U}^*\oplus{\cal O}(1)$ & IGr(3,8) \\
   
\end{tabular}
\caption{Threefolds in symplectic Grassmannians, see Theorem \ref{thm3class}}
\label{tablethreesym}
\end{table}

\begin{table}
\small
\begin{tabular}{lll}
   case & bundle ${\cal F}$ & OGr(k,n) \\
   \hline
   (ob1) & ${\cal T}_{+\frac{1}{2}}(1)\oplus S^2{\cal U}^*\oplus{\cal O}(1)$ & OGr(2,9) \\
   (ob2) & ${\cal T}_{+\frac{1}{2}}(1) \oplus {\cal O}(1)^{\oplus 4}$ & OGr(2,9) \smallskip \\
   
   (ob3) & ${\cal T}_{+\frac{1}{2}}(1)\oplus S^2{\cal U}^*\oplus{\cal O}(1)$ & OGr(2,11) \\
   (ob4) & ${\cal T}_{+\frac{1}{2}}(1) \oplus {\cal O}(1)^{\oplus 4}$ & OGr(2,11) \\
   \hline
   
   (ox1) & ${\cal T}_{+\frac{1}{2}}(1) \oplus {\cal T}_{+\frac{1}{2}}(2)$ & OGr(2,7) \\
   (ox2) & ${\cal T}_{+\frac{1}{2}}(1) \otimes {\cal U}^*$ & OGr(2,7) \\
   (ox3) & ${\cal T}_{+\frac{1}{2}}(1) \oplus {\cal O}(1)\oplus {\cal O}(2)$ & OGr(2,7) \\
   (ox4) & ${\cal T}_{+\frac{1}{2}}(1) \oplus {\cal U}^*(1)$ & OGr(2,7) \\
   (ox5) & $S^2{\cal T}_{+\frac{1}{2}}(1) \oplus {\cal O}(1)$ & OGr(2,7)  \smallskip \\
   
   (ox6) & ${\cal T}_{+\frac{1}{2}}(1)^{\oplus 4} \oplus {\cal O}(1)$ & OGr(3,9) \\
   (ox7) & ${\cal T}_{+\frac{1}{2}}(1)^{\oplus 3} \oplus \Lambda^2{\cal U}^*$ & OGr(3,9) \\
   \hline
   
   (oy2) & ${\cal O}(\frac{1}{2})\oplus{\cal O}(\frac{3}{2})\oplus{\cal O}(1)$ & OGr(3,7) \smallskip \\
   
   (oy3) & ${\cal O}(\frac{1}{2})^{\oplus 6} \oplus {\cal O}(1)$ & OGr(4,9) \smallskip \\
   
   (oy1) & ${\cal O}(\frac{1}{2}) \oplus \Lambda^2{\cal U}^*$ & OGr(5,11) \\
   
\end{tabular}
\caption{Threefolds in odd orthogonal Grassmannians, see Theorem \ref{thm3class}}
\label{tablethreeodd}
\end{table}

\begin{table}
\small
\begin{tabular}{lll}
   case & bundle ${\cal F}$ & OGr(k,n) \\
   \hline
   (ob5) & ${\cal T}_{+\frac{1}{2}}(1)^{\oplus 2} \oplus{\cal O}(2)\oplus{\cal O}(1)$ & OGr(2,10) \\
   (ob6) & ${\cal T}_{+\frac{1}{2}}(1)^{\oplus 2} \oplus{\cal U}^*(1)$ & OGr(2,10) \\
   (ob7) & ${\cal T}_{+\frac{1}{2}}(1)\oplus{\cal T}_{-\frac{1}{2}}(1) \oplus{\cal O}(2)\oplus{\cal O}(1)$ & OGr(2,10) \\
   (ob8) & ${\cal T}_{+\frac{1}{2}}(1)\oplus{\cal T}_{-\frac{1}{2}}(1) \oplus{\cal U}^*(1)$ & OGr(2,10) \smallskip \\
   
   (ob9) & ${\cal T}_{+\frac{1}{2}}(1) \oplus{\cal O}(2)\oplus{\cal O}(1)$ & OGr(2,14) \\
   (ob10) & ${\cal T}_{+\frac{1}{2}}(1) \oplus{\cal U}^*(1)$ & OGr(2,14) \\
   \hline
   
   (ow1) & ${\cal T}_{+\frac{1}{2}}(2)\oplus{\cal T}_{+\frac{1}{2}}(1)^{\oplus 2}$ & OGr(2,8) \\
   (ow2) & ${\cal T}_{+\frac{1}{2}}(2)\oplus{\cal T}_{-\frac{1}{2}}(1)^{\oplus 2}$ & OGr(2,8) \\
   (ow3) & ${\cal T}_{+\frac{1}{2}}(2)\oplus{\cal T}_{+\frac{1}{2}}(1)\oplus{\cal T}_{-\frac{1}{2}}(1)$ & OGr(2,8) \\
   (ow4) & $({\cal T}_{+\frac{1}{2}}(1)\otimes {\cal U}^*) \oplus{\cal T}_{+\frac{1}{2}}(1)$ & OGr(2,8) \\
   (ow5) & $({\cal T}_{+\frac{1}{2}}(1)\otimes {\cal U}^*) \oplus{\cal T}_{-\frac{1}{2}}(1)$ & OGr(2,8) \\
   (ow6) & ${\cal T}_{+\frac{1}{2}}(1)^{\oplus 2} \oplus{\cal O}(1) \oplus {\cal O}(2)$ & OGr(2,8) \\
   (ow7) & ${\cal T}_{+\frac{1}{2}}(1)\oplus {\cal T}_{-\frac{1}{2}}(1) \oplus{\cal O}(1) \oplus {\cal O}(2)$ & OGr(2,8) \\
   (ow8) & ${\cal T}_{+\frac{1}{2}}(1)^{\oplus 2} \oplus{\cal U}^*(1)$ & OGr(2,8) \\
   (ow9) & ${\cal T}_{+\frac{1}{2}}(1)\oplus {\cal T}_{-\frac{1}{2}}(1) \oplus{\cal U}^*(1)$ & OGr(2,8) \\
   (ow10) & ${\cal T}_{+\frac{1}{2}}(1)\oplus S^2{\cal U}^*\oplus {\cal O}(1)$ & OGr(2,8) \\
   (ow11) & ${\cal T}_{+\frac{1}{2}}(1)\oplus (1,1;1;1) \oplus{\cal O}(1)$ & OGr(2,8) \\
   (ow12) & ${\cal T}_{+\frac{1}{2}}(1)\oplus (1,1;1;-1) \oplus{\cal O}(1)$ & OGr(2,8) \\
   (ow13) & ${\cal T}_{+\frac{1}{2}}(1) \oplus({\cal U}^{\perp}/ {\cal U})(1)$ & OGr(2,8) \\
   (ow14) & ${\cal T}_{+\frac{1}{2}}(1) \oplus{\cal O}(1)^{\oplus 4}$ & OGr(2,8) \smallskip \\
   
   (ow15) & ${\cal T}_{+\frac{1}{2}}(1)^{\oplus 6}$ & OGr(3,10) \\
   (ow16) & ${\cal T}_{+\frac{1}{2}}(1)^{\oplus 5} \oplus {\cal T}_{-\frac{1}{2}}(1)$ & OGr(3,10) \\
   (ow17) & ${\cal T}_{+\frac{1}{2}}(1)^{\oplus 4} \oplus {\cal T}_{-\frac{1}{2}}(1)^{\oplus 2}$ & OGr(3,10) \\
   (ow18) & ${\cal T}_{+\frac{1}{2}}(1)^{\oplus 3} \oplus {\cal T}_{-\frac{1}{2}}(1)^{\oplus 3}$ & OGr(3,10) \\
   \hline
   
   (oz1) & ${\cal O}(\frac{1}{2})^{\oplus 2} \oplus{\cal O}(2)$ & OGr(4,8) \smallskip \\
   
   (oz2) & ${\cal O}(\frac{1}{2}) \oplus{\cal O}(2)\oplus \Lambda^4{\cal U}^*(-\frac{1}{2})$ & OGr(5,10) \\
   (oz3) & ${\cal O}(\frac{3}{2}) \oplus{\cal O}(1)\oplus \Lambda^4{\cal U}^*(-\frac{1}{2})$ & OGr(5,10) \\
   (oz4) & ${\cal O}(\frac{1}{2})^{\oplus 6} \oplus {\cal O}(1)$ & OGr(5,10) \smallskip \\
   
   (oz5) & ${\cal O}(\frac{1}{2})^{\oplus 6}\oplus \Lambda^5{\cal U}^*(-\frac{1}{2})$ & OGr(6,12) \\
   
\end{tabular}
\caption{Threefolds in even orthogonal Grassmannians, see Theorem \ref{thm3class}}
\label{tablethreeeven}
\end{table}

\begin{table}
\small
\begin{tabular}{lll}
   case & bundle ${\cal F}$ & OGr(n-1,2n)  \\
   \hline
   (oe1) & ${\cal L}_-(2) \oplus {\cal L}_+$ & OGr(2,6) \\
   (oe2) & ${\cal L}_-^{\otimes 3} \oplus {\cal L}_+^{\otimes 3}$ & OGr(2,6) \\
   (oe3) & ${\cal L}_-(1) \oplus {\cal L}_+(1)$ & OGr(2,6) \\
   (oe4) & ${\cal L}_-^{\otimes 2}(1) \oplus {\cal L}_+^{\otimes 2}$ & OGr(2,6) \smallskip \\
   
   (oe5) & ${\cal L}_-^{\oplus 3} \oplus ({\cal U}^*\otimes{\cal L}_+)$ & OGr(3,8) \\
   (oe6) & ${\cal O}(1)^{\oplus 2}\oplus {\cal L}_-^{\oplus 2} \oplus {\cal L}_+^{\oplus 2}$ & OGr(3,8) \\
   (oe7) & ${\cal L}_-^{\otimes 2} \oplus {\cal L}_+^{\otimes 2}\oplus {\cal L}_-^{\oplus 2} \oplus {\cal L}_+^{\oplus 2}$ & OGr(3,8) \\
   (oe8) & ${\cal L}_-(1)\oplus {\cal L}_-^{\oplus 2} \oplus {\cal L}_+^{\oplus 3}$ & OGr(3,8) \\
   (oe9) & ${\cal O}(1)\oplus {\cal L}_-^{\otimes 2}\oplus {\cal L}_- \oplus {\cal L}_+^{\oplus 3}$ & OGr(3,8) \\
   (oe10) & $({\cal L}_-^{\otimes 2})^{\oplus 2} \oplus {\cal L}_+^{\oplus 4}$ & OGr(3,8) \\
   (oe11) & $\Lambda^2{\cal U}^*\oplus {\cal O}(1)\oplus {\cal L}_- \oplus {\cal L}_+$ & OGr(3,8) \\
   (oe12) & $\Lambda^2{\cal U}^*\oplus {\cal L}_-^{\otimes 2} \oplus {\cal L}_+^{\oplus 2}$ & OGr(3,8) \\
   (oe13) & ${\cal L}_-^{\oplus 4} \oplus {\cal L}_+ \oplus {\cal L}_+^{\otimes 3}$ & OGr(3,8) \\
   
\end{tabular}
\caption{Threefolds in OGr(n-1,2n), see Theorem \ref{theorem3}}
\label{tablethreen-1}
\end{table}

\begin{table}
\small
\begin{tabular}{lllll}
   case & bundle ${\cal F}$ & Gr(k,n) & $\chi({\cal O}_Y)$ & degree\\
   \hline
   (a) & complete intersection of hypersurfaces & $\mathbb{P}^n$ & & \\
   \hline
   (b7) & ${\cal O}(1)\oplus{\cal O}(3)$ & Gr(2,4) & 2 & 6 \\
   (b8) & ${\cal O}(2)^{\oplus 2}$ & Gr(2,4) & 2 & 8 \smallskip \\
   
   (b1) & ${\cal Q}(1)\oplus{\cal O}(1)$ & Gr(2,5) & 2 & 14 \\
   (b2) & $\Lambda^2{\cal Q}\oplus{\cal O}(3)$ & Gr(2,5) & 2 & 6 \\
   (b9) & $S^2{\cal U}^*\oplus{\cal O}(2)$ & Gr(2,5) & 2 & 16 \\
   M.(b10) & ${\cal O}(1)^{\oplus 3}\oplus{\cal O}(2)$ & Gr(2,5) & 2 & 10 \smallskip \\
   
   (b3) & $\Lambda^3{\cal Q}\oplus{\cal O}(2)\oplus{\cal O}(1)$ & Gr(2,6) & 2 & 12 \\
   (b4) & $\Lambda^3{\cal Q}\oplus{\cal U}^*(1)$ & Gr(2,6) & 2 & 14 \\
   (b12) & $S^2{\cal U}^*\oplus{\cal O}(1)^{\oplus 3}$ & Gr(2,6) & 2 & 20 \\  
   M.(b13) & ${\cal O}(1)^{\oplus 6}$ & Gr(2,6) & 2 & 14 \\
   (b11) & $(S^2{\cal U}^*)^{\oplus 2}$ & Gr(2,6) & 0 & 32 \smallskip \\
   
   (b5) & $S^2{\cal U}^*\oplus\Lambda^4{\cal Q}$ & Gr(2,7) & 2 & 24 \\
   (b6) & ${\cal O}(1)^{\oplus 3}\oplus\Lambda^4{\cal Q}$ & Gr(2,7) & 2 & 18 \\
   \hline
   
   (c2) & $\Lambda^2{\cal U}^*\oplus \Lambda^2{\cal Q} \oplus {\cal O}(2)$ & Gr(3,6) & 2 & 12 \\
   (c4) & $S^2{\cal U}^* \oplus {\cal O}(2)$ & Gr(3,6) & 4 & 32 \\
   (c5) & $(\Lambda^2{\cal U}^*)^{\oplus 2} \oplus {\cal O}(2)$ & Gr(3,6) & 2 & 12\\
   M.(c6) & $\Lambda^2{\cal U}^* \oplus {\cal O}(1)^{\oplus 4}$ & Gr(3,6) & 2 & 16 \smallskip \\
   
   (c3) & $S^2{\cal U}^* \oplus \Lambda^3{\cal Q}$ & Gr(3,7) & 2 & 48 \\
   M.(c9) & $\Lambda^3{\cal Q} \oplus (\Lambda^2{\cal U}^*)^{\oplus 2}$ & Gr(3,7) & 2 & 24 \\
   (c7) & $S^2{\cal U}^* \oplus \Lambda^2{\cal U}^* \oplus {\cal O}(1)$ & Gr(3,7) & 2 & 48 \\
   (c8) & $(\Lambda^2{\cal U}^*)^{\oplus 3} \oplus {\cal O}(1)$ & Gr(3,7) & 2 & 22 \smallskip \\
   
   (c1) & $\Lambda^3{\cal Q} \oplus \Lambda^2{\cal U}^*$ & Gr(3,8) & 2 & 36 \\
   \hline
   
   (d1) & $S^2{\cal U}^* \oplus \Lambda^3{\cal U}^*$ & Gr(4,8) & 4 & 96 \\
   (d1.1) & $S^2{\cal U}^* \oplus \Lambda^3{\cal Q}$ & Gr(4,8) & 4 & 96 \\
   (d2) & $(\Lambda^2{\cal U}^*)^{\oplus 2} \oplus {\cal O}(1)^{\oplus 2}$ & Gr(4,8) & 2 & 24 \\
   (d2.1) & $\Lambda^2{\cal U}^*\oplus \Lambda^2{\cal Q} \oplus {\cal O}(1)^{\oplus 2}$ & Gr(4,8) & 2 & 24 \smallskip \\
   
   M.(d3) & $(\Lambda^2{\cal U}^*)^{\oplus 3}$ & Gr(4,9) & 2 & 38 \\
   
\end{tabular}
\caption{Surfaces in ordinary Grassmannians, see Theorem \ref{thm2ord}}
\label{tabletwoord}
\end{table}

\begin{table}
\small
\begin{tabular}{lllll}
   case & bundle ${\cal F}$ & IGr(k,2n) & $\chi({\cal O}_Y)$ & degree \\
   \hline
   (sb1) & $({\cal U}^{\perp}/ {\cal U})(1)\oplus{\cal U}^*\oplus{\cal O}(2)$ & IGr(2,6) & 2 & 12 \\
   (sb2) & $({\cal U}^{\perp}/ {\cal U})(1)\oplus S^2{\cal U}^*$ & IGr(2,6) & 2 & 24 \\
   (sb3) & $({\cal U}^{\perp}/ {\cal U})(1)^{\oplus 2}\oplus{\cal O}(1)$ & IGr(2,6) & 2 & 24 \\
   (sb4) & $({\cal U}^{\perp}/ {\cal U})(1)\oplus{\cal O}(1)^{\oplus 3}$ & IGr(2,6) & 2 & 18 \\
   
\end{tabular}
\caption{Surfaces in symplectic Grassmannians, see Theorem \ref{thm2class}}
\label{tabletwosym}
\end{table}

\begin{table}
\small
\begin{tabular}{lllll}
   case & bundle ${\cal F}$ & OGr(k,n) & $\chi({\cal O}_Y)$ & degree \\
   \hline
   (ob1) & ${\cal T}_{+\frac{1}{2}}(1)^{\oplus 2}\oplus{\cal O}(2)$ & OGr(2,9) & 2 & 12 \smallskip \\
   
   (ob2) & ${\cal T}_{+\frac{1}{2}}(1)\oplus{\cal O}(2)$ & OGr(2,13) & 4 & 24 \\
   \hline
   
   (ox1) & ${\cal T}_{+\frac{1}{2}}(1)^{\oplus 2} \oplus {\cal O}(2)$ & OGr(2,7) & 2 & 12 \\
   (ox2) & ${\cal T}_{+\frac{1}{2}}(1) \oplus S^2{\cal U}^*$ & OGr(2,7) & 2 & 24 \\
   (ox3) & ${\cal T}_{+\frac{1}{2}}(1) \oplus S^2{\cal T}_{+\frac{1}{2}}(1)$ & OGr(2,7) & 2 & 24 \\
   (ox4) & ${\cal T}_{+\frac{1}{2}}(1) \oplus {\cal O}(1)^{\oplus 3}$ & OGr(2,7) & 2 & 18  \smallskip \\
   
   M.(ox5) & ${\cal T}_{+\frac{1}{2}}(1)^{\oplus 5}$ & OGr(3,9) & 2 & 34 \\
   \hline
   
   (oy2) & ${\cal O}(\frac{3}{2}) \oplus {\cal O}(\frac{1}{2})^{\oplus 3}$ & OGr(3,7) & 2 & 6 \\
   (oy3) & ${\cal U}^*(\frac{1}{2}) \oplus {\cal O}(\frac{1}{2})$ & OGr(3,7) & 2 & 12 \\
   (oy4) & ${\cal O}(\frac{1}{2})^{\oplus 2} \oplus {\cal O}(1)^{\oplus 2}$ & OGr(3,7) & 2 & 8 \smallskip \\
   
   (oy1) & ${\cal O}(\frac{1}{2})^{\oplus 2} \oplus \Lambda^2{\cal U}^*$ & OGr(4,9) & 2 & 24 \\
   M.(oy5) & ${\cal O}(\frac{1}{2})^{\oplus 8}$ & OGr(4,9) & 2 & 12 \\
   
\end{tabular}
\caption{Surfaces in odd orthogonal Grassmannians, see Theorem \ref{thm2class}}
\label{tabletwoodd}
\end{table}

\begin{table}
\small
\begin{tabular}{lllll}
   case & bundle ${\cal F}$ & OGr(k,n) & $\chi({\cal O}_Y)$ & degree \\
   \hline
   (ob3) & ${\cal T}_{+\frac{1}{2}}(2) $ & Quadric in $\mathbb{P}^7$ & 2 & 12 \smallskip \\
   
   (ob4) & ${\cal T}_{+\frac{1}{2}}(1)^{\oplus 2} \oplus{\cal O}(1)^{\oplus 3}$ & OGr(2,10) & 2 & 18 \\
   (ob4.2) & ${\cal T}_{+\frac{1}{2}}(1) \oplus{\cal T}_{-\frac{1}{2}}(1) \oplus{\cal O}(1)^{\oplus 3}$ & OGr(2,10) & 2 & 20 \\
   (ob5) & ${\cal T}_{+\frac{1}{2}}(1)^{\oplus 2} \oplus S^2{\cal U}^*$ & OGr(2,10) & 2 & 24 \\
   (ob5.2) & ${\cal T}_{+\frac{1}{2}}(1) \oplus{\cal T}_{-\frac{1}{2}}(1)\oplus S^2{\cal U}^*$ & OGr(2,10) & 2 & 24 \smallskip \\
   
   (ob6) & ${\cal T}_{+\frac{1}{2}}(1) \oplus{\cal O}(1)^{\oplus 3}$ & OGr(2,14) & 4 & 36  \\
   (ob7) & ${\cal T}_{+\frac{1}{2}}(1) \oplus S^2{\cal U}^*$ & OGr(2,14) & 4 & 48 \\
   \hline
   
   (ow1) & ${\cal T}_{+\frac{1}{2}}(1)^{\oplus 2} \oplus{\cal O}(1)^{\oplus 3}$ & OGr(2,8) & 2 & 20 \\
   (ow2) & ${\cal T}_{+\frac{1}{2}}(1)^{\oplus 2} \oplus (1,1;1;1)$ & OGr(2,8) & 0 & 32 \\
   (ow3) & ${\cal T}_{+\frac{1}{2}}(1)^{\oplus 2} \oplus S^2{\cal U}^*$ & OGr(2,8) & 2 & 24 \\
   (ow4) & $({\cal T}_{+\frac{1}{2}}(1)^{\oplus 3} \oplus{\cal O}(2)$ & OGr(2,8) & 2 & 16 \\
   (ow1.2) & ${\cal T}_{+\frac{1}{2}}(1) \oplus{\cal T}_{-\frac{1}{2}}(1) \oplus{\cal O}(1)^{\oplus 3}$ & OGr(2,8) & 2 & 18 \\
   (ow2.2) & ${\cal T}_{+\frac{1}{2}}(1) \oplus{\cal T}_{-\frac{1}{2}}(1) \oplus (1,1;1;1)$ & OGr(2,8) & 2 & 24 \\
   (ow2.3) & ${\cal T}_{-\frac{1}{2}}(1)^{\oplus 2} \oplus (1,1;1;1)$ & OGr(2,8) & 2 & 24 \\
   (ow3.2) & ${\cal T}_{+\frac{1}{2}}(1) \oplus{\cal T}_{-\frac{1}{2}}(1)\oplus S^2{\cal U}^*$ & OGr(2,8) & 2 & 24 \\
   (ow4.2) & ${\cal T}_{+\frac{1}{2}}(1)^{\oplus 2} \oplus{\cal T}_{-\frac{1}{2}}(1) \oplus{\cal O}(2)$ & OGr(2,8) & 2 & 12 \\
   \hline
   
   (oz3) & ${\cal U}^*(\frac{1}{2})$ & OGr(4,8) & 2 & 12 \smallskip \\
   
   (oz4) & ${\cal O}(\frac{1}{2})^{\oplus 2} \oplus{\cal O}(\frac{3}{2})\oplus\Lambda^4{\cal U}^*(-\frac{1}{2})$ & OGr(5,10) & 2 & 6 \\
   (oz5) & ${\cal O}(1)^{\oplus 2} \oplus{\cal O}(\frac{1}{2})\oplus\Lambda^4{\cal U}^*(-\frac{1}{2})$ & OGr(5,10) & 2 & 8 \smallskip \\
   
   (oz7) & ${\cal O}(1)\oplus\Lambda^5{\cal U}^*(-\frac{1}{2})^{\oplus 2}$ & OGr(6,12) & 2 & 12 \\
   
\end{tabular}
\caption{Surfaces in even orthogonal Grassmannians, see Theorem \ref{thm2class}}
\label{tabletwoeven}
\end{table}

\begin{table}
\small
\begin{tabular}{lllllll}
   case & bundle ${\cal F}$ & OGr(n-1,2n) & $\chi({\cal O}_Y)$ & deg(${\cal O}(1)$) & deg(${\cal L}_+$) & deg(${\cal L}_-$) \\
   \hline
   (oe1) & ${\cal L}_-^{\otimes 2} \oplus ({\cal U}^*\otimes {\cal L}_+)$ & OGr(2,6) & 2 & 26 & 4 & 6 \\
   (oe2) & ${\cal L}_+\oplus {\cal L}_-^{\otimes 3} \oplus {\cal L}_+^{\otimes 2}$ & OGr(2,6) & 2 & 18 & 0 & 6 \\
   (oe3) & ${\cal L}_-(1)\oplus {\cal L}_- \oplus {\cal L}_+^{\otimes 2}$ & OGr(2,6) & 2 & 18 & 4 & 2 \\
   (oe4) & ${\cal O}(1)\oplus {\cal L}_-(1) \oplus {\cal L}_+$ & OGr(2,6) & 2 & 16 & 2 & 4 \\
   (oe5) & ${\cal L}_-\oplus {\cal L}_- \oplus {\cal L}_+^{\otimes 2}(1)$ & OGr(2,6) & 2 & 10 & 4 & 0 \\
   (oe6) & ${\cal O}(2)\oplus {\cal L}_- \oplus {\cal L}_+$ & OGr(2,6) & 2 & 12 & 2 & 2 \\
   (oe7) & ${\cal O}(1)\oplus {\cal L}_-^{\otimes 2} \oplus {\cal L}_+^{\otimes 2}$ & OGr(2,6) & 2 & 24 & 4 & 4 \smallskip \\
   
   (oe8) & $\Lambda^2{\cal U}^*\oplus {\cal L}_-^{\oplus 2} \oplus {\cal L}_+^{\oplus 2}$ & OGr(3,8) & 2 & 36 & 8 & 8 \\
   (oe9) & ${\cal O}(1)\oplus {\cal L}_-^{\oplus 3} \oplus {\cal L}_+^{\oplus 3}$ & OGr(3,8) & 2 & 28 & 6 & 6 \\
   (oe10) & ${\cal L}_+^{\otimes 2}\oplus {\cal L}_-^{\oplus 4} \oplus {\cal L}_+^{\oplus 2}$ & OGr(3,8) & 2 & 28 & 8 & 4 \\
   
\end{tabular}
\caption{Surfaces in OGr(n-1,2n), see Theorem \ref{theorem2}}
\label{tabletwon-1}
\end{table}

\FloatBarrier

\bibliographystyle{alpha}
\bibliography{bibliiot}

\begin{thebibliography}{{Muk}06}

\bibitem[BD85]{BeauvilleDonagi}
A.~{Beauville} and R.~{Donagi}.
\newblock {The variety of lines of a cubic fourfold}.
\newblock {\em C. R. Acad. Sci.Paris 301 s\'erie I, (1985) 703-706}, 1985.

\bibitem[Bea83]{Beauville}
Arnaud Beauville.
\newblock {V}ariétés {K}ähleriennes dont la première classe de {C}hern est
  nulle.
\newblock {\em J. Differential Geom.}, 18(4):755--782, 1983.

\bibitem[BKT09]{Buchtam}
A.~S. Buch, A.~Kresch, and H.~Tamvakis.
\newblock {Q}uantum {P}ieri rules for isotropic {G}rassmannians.
\newblock {\em Inventiones mathematicae}, 178(2):345--405, 2009.

\bibitem[{Bog}74]{Bogomolov}
F.A. {Bogomolov}.
\newblock {\"Uber die Zerlegung K\"ahlerscher Mannigfaltigkeiten mit trivialer
  kanonischer Klasse.}
\newblock {\em {Mat. Sb., Nov. Ser.}}, 93:573--575, 1974.

\bibitem[Bot57]{Bott}
Raoul Bott.
\newblock Homogeneous vector bundles.
\newblock {\em Annals of Mathematics}, 66(2):203--248, 1957.

\bibitem[BP05]{Brion}
M.~Brion and P.~Pragacz.
\newblock {\em {L}ectures on the {G}eometry of {F}lag {V}arieties}, pages
  33--85.
\newblock Birkh{\"a}user Basel, Basel, 2005.

\bibitem[Dem73]{Demazure}
Michel Demazure.
\newblock Invariants symétriques entiers des groupes de {W}eyl et torsion.
\newblock {\em Inventiones mathematicae}, 21:287--302, 1973.

\bibitem[DV10]{VoisinDebarre}
O.~{Debarre} and C.~{Voisin}.
\newblock {Hyper-K\"ahler fourfolds and Grassmann geometry}.
\newblock {\em J. reine angew. Math. 649}, pages 63--87, 2010.

\bibitem[FH04]{Fulton}
W.~{Fulton} and J.~{Harris}.
\newblock {\em {Representation Theory - A First Course}}, volume 129.
\newblock Springer-Verlag New York, 2004.

\bibitem[FP98]{FultonPragacz}
W.~Fulton and P.~Pragacz.
\newblock {\em {S}chubert varieties and degeneracy loci}.
\newblock Lecture notes in mathematics. Springer, 1998.

\bibitem[GS]{Macaulay}
Daniel~R. Grayson and Michael~E. Stillman.
\newblock Macaulay2, a software system for research in algebraic geometry.
\newblock Available at \url{http://www.math.uiuc.edu/Macaulay2/}.

\bibitem[IIM16]{Japonais}
D.~{Inoue}, A.~{Ito}, and M.~{Miura}.
\newblock {Complete intersection Calabi--Yau manifolds with respect to
  homogeneous vector bundles on Grassmannians}.
\newblock {\em ArXiv e-prints}, July 2016.

\bibitem[Kü95]{Kuchle1995}
Oliver Küchle.
\newblock On {F}ano 4-folds of index 1 and homogeneous vector bundles over
  {G}rassmannians.
\newblock {\em Mathematische Zeitschrift}, 218(4):563--576, 1995.

\bibitem[{Muk}88]{Mukai10}
S.~{Mukai}.
\newblock {Curves, $K3$ surfaces and Fano $3$-folds of genus $\leq 10$}.
\newblock {\em Algebraic geometry and commutative algebra ; in honor of
  Masayoshi Nagata}, 1988.

\bibitem[{Muk}92]{Mukai1820}
S.~{Mukai}.
\newblock {Polarized $K3$ surfaces of genus $18$ and $20$}.
\newblock {\em Complex Projective Geometry, Cambridge Univ. Press}, pages
  264--276, 1992.

\bibitem[{Muk}06]{Mukai13}
S.~{Mukai}.
\newblock {Polarized $K3$ surfaces of genus thirteen}.
\newblock {\em Advanced Studies in Pure Mathematics 45 - Moduli Spaces and
  Arithmetic Geometry (Kyoto, 2004)}, pages 315--326, 2006.

\bibitem[Ott95]{Ottaviani}
Giorgio Ottaviani.
\newblock Rational homogeneous varieties.
\newblock {\em Lecture notes for the summer school in Algebraic Geometry in
  Cortona}, 1995.

\bibitem[{Rei}72]{Reid}
Miles {Reid}.
\newblock {The complete intersection of two or more quadrics}.
\newblock {\em PHD Dissertation}, 1972.

\bibitem[Tam05]{Tamvakis}
Harry Tamvakis.
\newblock {\em Quantum cohomology of isotropic {G}rassmannians}, pages
  311--338.
\newblock Birkh{\"a}user Boston, Boston, MA, 2005.

\bibitem[vCL]{Lie}
M.~A.~A. {von Leeuwen}, A.~M. {Cohen}, and B.~{Lisser}.
\newblock Lie, a software package for {L}ie group theoretical computations.
\newblock Available at \url{http://wwwmathlabo.univ-poitiers.fr/~maavl/LiE/}.

\end{thebibliography}

\end{document}